\documentclass[11pt,a4paper]{amsart}
\usepackage{amssymb,amsmath}
\usepackage{verbatim}
\usepackage{booktabs}
\usepackage{enumerate}
\usepackage{hyperref}
\usepackage{graphicx}
\usepackage{subfigure}

\makeatletter
\newtheorem{theorem}{Theorem} 
\newtheorem{lemma}[theorem]{Lemma}

\newtheorem{assumption}[theorem]{Assumption}
\theoremstyle{remark}
\newtheorem{remark}{Remark}
\numberwithin{remark}{section}
\numberwithin{theorem}{section}
\numberwithin{equation}{section}

\newcommand{\tri}{\mathcal{T}}
\newcommand{\triH}{\tri_H}

\newcommand{\dx}{\operatorname*{d}\hspace{-0.3ex}x}
\newcommand{\ds}{\operatorname*{d}\hspace{-0.3ex}s}

\newcommand{\dist}{\operatorname*{dist}}

\newcommand{\kernel}{\operatorname*{kernel}}
\newcommand{\support}{\operatorname*{supp}}

\newcommand{\N}{\mathcal{N}}

\newcommand{\R}{\mathbb{R}}
\newcommand{\C}{\mathbb{C}}
\newcommand{\IH}{\mathcal{I}_H}
\newcommand{\Inodal}{\mathcal{I}^{\operatorname{nodal}}_H}
\newcommand{\Ih}{\mathcal{I}^{\operatorname{nodal}}_h}

\newcommand{\Cint}{C_{\IH}}

\newcommand{\Col}{C_{\operatorname*{ol}}}
\newcommand{\Coll}[1]{C_{\operatorname*{ol},#1}}
\newcommand{\Cloc}[1]{C_{\operatorname*{loc},#1}}

\newcommand{\Cdec}{C_{\operatorname*{dec}}}

\newcommand{\Vp}{V_{H,\infty}}
\newcommand{\Vd}{V^*_{H,\infty}}
\newcommand{\Vpl}{V_{H,\ell}}
\newcommand{\Vplh}{V_{H,\ell,h}}
\newcommand{\Vdl}{V^*_{H,\ell}}
\newcommand{\Vdlh}{V^*_{H,\ell,h}}
\newcommand{\vdl}{v_{H,\ell}}
\newcommand{\vdlh}{v_{H,\ell,h}}
\newcommand{\vpl}{v_{H,\ell}}
\newcommand{\vp}{v_{H,\infty}}
\newcommand{\zdl}{z_{H,\ell}}
\newcommand{\zd}{z_{H,\infty}}
\newcommand{\upl}{u_{H,\ell}}
\newcommand{\uplh}{u_{H,\ell,h}}
\newcommand{\epl}{e_{H,\ell}}
\newcommand{\ep}{e}
\newcommand{\vd}{v_{H,\infty}}
\newcommand{\up}{u_{H,\infty}}

\renewcommand{\k}{\kappa}
\newcommand{\Cstab}{C_{\operatorname*{st}}}
\newcommand{\Cpstab}{C_{\operatorname*{st}}'}

\newcommand{\bstab}{{n}}
\newcommand{\Cor}{\mathcal{C}}
\newcommand{\Ca}{C_a}
\newcommand{\CC}{C_{\Cor}}
\newcommand{\CPH}{C_{\Pi_H}}
\newcommand{\Iloc}{\mathcal{I}_H^{-1,\operatorname{loc}}}

\title[Eliminating pollution by subscale correction]{Eliminating the pollution effect in Helmholtz problems by local subscale correction}

\author{Daniel Peterseim}
\address[D. Peterseim]{Rheinische Friedrich-Wilhelms-Universit\"at Bonn, Institute for Numerical Simulation, Wegelerstr. 6, 53115 Bonn, Germany}
\email{peterseim@ins.uni-bonn.de}

\date{\today}



\begin{document}

\begin{abstract}
We introduce a new Petrov-Galerkin multiscale method for the numerical approximation of the Helmholtz equation with large wave number $\k$ in bounded domains in $\mathbb{R}^d$. The discrete trial and test spaces are generated from standard mesh-based finite elements by local subscale correction in the spirit of numerical homogenization. The precomputation of the correction involves the solution of coercive cell problems on localized subdomains of size $\ell H$; $H$ being the mesh size and $\ell$ being the oversampling parameter. If the mesh size and the oversampling parameter are such that $H\k$ and $\log(\k)/\ell$ fall below some generic constants and if the cell problems are solved sufficiently accurate on some finer scale of discretization, then the method is stable and its error is proportional to $H$; pollution effects are eliminated in this regime.
\end{abstract}

\maketitle

\section{Introduction}
The numerical solution of the Helmholtz equation by the finite element method or related schemes in the regime of large wave numbers is still among the most challenging tasks of computational partial differential equations. The highly oscillatory nature of the solution plus a wave number dependent pollution effect puts very restrictive assumptions on the smallness of the underlying mesh. Typically, this condition is much stronger than the minimal requirement for a meaningful representation of highly oscillatory functions from approximation theory, that is, to have at least $5-10$ degrees of freedom per wave length and coordinate direction. 

The wave number dependent preasymptotic effect denoted as pollution or numerical dispersion is well understood by now and many attempts have been made to overcome or at least reduce it; see \cite{MR2219901,MR2551150,MR2813347,perugia,Zitelli20112406,dpg} among many others. 
However, for many standard methods, this is not possible in 2d or 3d \cite{BabSau}. A breakthrough in this context is the work of Melenk and Sauter \cite{MS10,MS11,parsania}. It shows that for certain model Helmholtz problems, the pollution effect can be suppressed in principle by simply coupling the polynomial degree $p$ of the Galerkin finite element space to the wave number $\k$ via the relation $p\approx\log\k$. Under this moderate assumption, the method is stable and quasi-optimal if the mesh size $H$ satisfies $H\k/p\lesssim 1$ (plus certain adaptive refinement towards non-smooth geometric features). It is worth noting that this result does not require the analyticity of the solution but only $W^{2,2}$-regularity and, thus, partially explains the common sense that higher-order methods are less sensitive to pollution. However, for less regular solutions as they appear for the scattering of waves from non-smooth objects or inhomogeneities of the material, the result is not directly applicable and the existence of a pollution-free discretization scheme remained open.

Scale-dependent preasymptotic effects are also observed in simpler diffusion problems with highly oscillatory diffusion tensor and numerical homogenization provides techniques to avoid those effects. 
Numerical homogenization (or upscaling) refers to a class of multiscale methods for the efficient approximation on coarse meshes that do not resolve the coefficient oscillations. A novel method for this problem was recently introduced in \cite{MP14} and further generalized in \cite{EGMP13,HMP13,HP12,HMP14}. The method is based on localizable orthogonal decompositions (LOD) into a low-dimensional coarse space (where we are looking for the approximation) and a high-dimensional remainder space. Some selectable quasi-interpolation operator serves as the basis of the decompositions. The coarse space is spanned by some precomputable basis functions with local support. The method provides text book convergence independent of the variations of the coefficient and without any preasymptotic effects under fairly general assumptions on the diffusion coefficient; periodicity or scale separation are not required.

This paper adapts the multiscale method of \cite{MP14} to cure pollution in the numerical approximation of the Helmholtz problem. To deal with the lack of hermitivity we will propose a Petrov-Galerkin version of the method (although this is not essential). We will construct a finite-dimensional trial space and corresponding test space for the approximation of the unknown solution $u$. The trial and test spaces are generated from standard mesh-based finite elements by local subscale correction. The precomputation of the correction involves the solution of $\mathcal{O}(H^{-d})$ coercive cell problems on localized subdomains of size $\ell H$; $H$ being the mesh size and $\ell$ being the adjustable oversampling parameter. If the data of the problem (domain, boundary condition, force term) allows for polynomial-in-$\k$ bounds of the solution operator and if the mesh size and the oversampling parameter of the method are such that the resolution condition $H\k\lesssim 1$ and the oversampling condition $\ell\gtrsim \log(\k)$ are satisfied, then the method is stable. If, moreover, $\ell\gtrsim \log(|H|)\log(\k) $ then the method satisfies the error estimate
\begin{equation*}
 \k\|u-u_{H,\ell}\|_{L^2(\Omega)}+\|\nabla(u-u_{H,\ell})\|_{L^2(\Omega)}\lesssim H\|f\|_{L^2(\Omega)}
\end{equation*}
with some generic constant $C>0$ of $\k$. For a fairly large class of Helmholtz problems, including the acoustic scattering from convex non-smooth objects, this result shows that pollution effects can be suppressed under the quasi-minimal resolution condition $H\k\leq \mathcal{O}(1)$ at the price of a moderate increase of the inter-element communication, i.e., logarithmic-in-$\k$ oversampling. Using a terminology from finite difference methods, this means that the stencil is moderately enlarged. The complexity overhead due to oversampling is comparable with that of \cite{MS10,MS11}, where instead of increasing the inter-element communication, the number of degrees of freedom per element is increased via the polynomial degree which is coupled to $\log\k$ in a similar way. 

While \cite{BabSau} shows that pollution cannot be avoided with a fixed stencil, the result shows that already a logarithmic-in-$\k$ growths of the stencil  can suffice to eliminate pollution. Although the result is constructive, its practical relevance for actual computations is not immediately clear in any case. The multiscale method presented in this paper requires the precomputation of local cell problems on a finer scale of numerical resolution. The suitable choice of this fine scale depends on the stability properties of the problem in the same way as standard finite element methods do. However, the local cell problems are independent and, though being Helmholtz problems, they are somewhat simpler because they are essentially low-frequency when compared to the characteristic length scale of the cell. Still, the worst-case (serial) complexity of the method can exceed the cost of a direct numerical simulation on the global fine scale. We expect a significant gain with respect to computational complexity in the following cases:
\begin{itemize}
 \item The precomputation can be reused several times, e.g., if the problem (with the same geometric setting and wave number) has to be solved for a large number of force terms or incident wave directions in the context of parameter studies, coupled problems or inverse problems.\vspace{1ex}
 \item The (local) periodicity of the computational mesh can be exploited so that the number of local problems can be reduced drastically.
\end{itemize}
We also expect that the redundancy of the local problems can be exploited in rather general unstructured meshes by modern techniques of model order reduction \cite{MR2430350,MR3247814}. However, this possibility requires a careful algorithmic design and error analysis which are beyond the scope of this paper and remain a future perspective of the method. A similar statement applies to the case of heterogeneous media. This application and the generalization of the method are very natural and straight forward. Though this case is not yet covered, previous work \cite{MP14,EGMP13,HMP13,HP12,HMP14} plus the analysis of this paper strongly indicate the potential of the method to treat high oscillations or jumps in the PDE coefficients and the pollution effect in one stroke.  

The remaining part of the paper is outlined as follows. Section~\ref{s:model} defines the model Helmholtz problem and recalls some of its fundamental properties. Section~\ref{s:pre} introduces standard finite element spaces and corresponding quasi-interpolation operators that will be the basis for 
the derivation of a prototypical multiscale method in Section~\ref{s:omd}. Sections~\ref{s:loc} and \ref{s:fine} will then turn this ideal approach into a feasible method including a rigorous stability and error analysis. Finally, Section~\ref{s:numexp} illustrate the performance of the method and one of its variants in numerical experiments.

\section{Model Helmholtz problem}\label{s:model} We consider the Helmholtz equation over a bounded Lipschitz domain $\Omega\subset\R^{d}$ ($d=1,2,3$),\begin{subequations}
\label{e:model}
\end{subequations}
\begin{equation}\label{e:modela}
  -\Delta u - \k^2 u = f\quad \text{in }\Omega,
 \tag{\ref{e:model}.a}
\end{equation}
along with mixed boundary conditions of Dirichlet, Neumann and Robin type
\begin{align}
  u &= 0\quad\text{on }\Gamma_D,\tag{\ref{e:model}.b}\label{e:modelb}\\
  \nabla u\cdot \nu &= 0\quad\text{on }\Gamma_N,\tag{\ref{e:model}.c}\label{e:modelc}\\
  \nabla u\cdot \nu - i\k u &= 0\quad\text{on }\Gamma_R.\tag{\ref{e:model}.d}\label{e:modeld}
\end{align}
Here, the wave number $\k$ is real and positive, $i$ denotes the imaginary unit and $f\in L^2(\Omega)$ (the space of complex-valued square-integrable functions over $\Omega$). In this paper, we assume that the boundary $\Gamma :=\partial \Omega $
consists of three components
\begin{equation*}
\partial\Omega =\overline{\Gamma _D\cup \Gamma _N\cup \Gamma _R},
\end{equation*}
where $\Gamma _D$, $\Gamma _N$ and $\Gamma _R$ are disjoint. We allow that $\Gamma _D$ or $\Gamma _N$  are empty but we assume that $\Gamma_R$ has a positive surface measure, 
\begin{equation}\label{e:GammaR}
 |\Gamma_R|>0.
\end{equation}
The vector $\nu$ denotes the unit normal vector that is outgoing from $\Omega$. To avoid overloading of the paper, we restrict ourselves to the case of homogeneous boundary conditions. Since inhomogeneous boundary data is very relevant for scattering problems, this case will be treated in the context of a numerical experiment in Section~\ref{ss:numexp2}.

Given the Sobolev space $W^{1,2}(\Omega)$ (the space of complex-valued square-integrable functions over $\Omega$ with square integrable weak gradient), we introduce the subspace 
\begin{equation*}
V:=\{v\in W^{1,2}(\Omega)\;\vert\; v=0\text{ on }\Gamma_D\} 
\end{equation*}
along with the $\k$-weighted norm 
\begin{equation*}
\|v\|_V:=\sqrt{\k^2\|v\|_\Omega^2+\|\nabla v\|_\Omega^2}, 
\end{equation*}
where $\|\cdot\|_\Omega$ denotes the $L^2$-norm over $\Omega$. The variational formulation of the boundary value problem~\eqref{e:model} seeks $u\in V$ such that, for all $v\in V$, 
\begin{equation}\label{e:modelvar}
 a(u,v) = (f,v)_\Omega,
\end{equation}
where the sesquilinear form $a:V\times V\rightarrow \C$ has the form
\begin{equation}\label{e:a}
 a(u,v):=(\nabla u,\nabla v)_\Omega-\k^2(u,v)_\Omega - i\k(u,v)_{\Gamma_R}.
\end{equation}
Here, $(\cdot,\cdot)_\Omega:=\int_\Omega u\cdot\bar v\dx$ abbreviates the canonical inner product of scalar or vector-valued $L^2(\Omega)$ functions and $(\cdot,\cdot)_{\Gamma_R}:=\int_{\Gamma_R} u\bar v\ds$ abbreviates the canonical inner product of $L^2(\Gamma_R)$ (the space of complex-valued square-integrable functions over $\Gamma_R$). The sesquilinear form $a$ is bounded, i.e.,
there is a constant $\Ca$ that depends only on $\Omega$ such that, for any $u,v\in V$,
\begin{equation}\label{e:acont}
 |a(u,v)|\leq\Ca \|u\|_V\|v\|_V.
\end{equation}

The presence of the impedance boundary condition \eqref{e:modeld} (cf. \eqref{e:GammaR}) ensures the well-posedness of problem \eqref{e:modelvar}, i.e., there exists some constant $\Cstab(\k)$ that may depend on $\k$ and also on $\Omega$ and the partition of the boundary into $\Gamma_D$, $\Gamma_N$ and $\Gamma_R$ such that, for any $f\in L^2(\Omega)$, the unique solution $u\in V$ of \eqref{e:modelvar} satisfies 
\begin{equation}\label{e:stab}
 \|u\|_V\leq \Cstab(\k)\|f\|_\Omega.
\end{equation}
However, the stability constant $\Cstab(\k)$ and its possible dependence on the wave number $\k$ are not known in general. Whenever we want to quantify its effect on some parts of the error analysis, we will assume (cf. Assumption~\ref{a:oversampling} below) that there are constants $\Cpstab>0$ and $\bstab \geq 0$ and $\k_0>0$ that may depend on $\Omega$ and the partition of the boundary into $\Gamma_D$, $\Gamma_N$ and $\Gamma_R$ such that, for any $\k\geq\k_0$, the stability constant $\Cstab(\k)$ of \eqref{e:stab} satisfies
\begin{equation}\label{e:stabpoly}
\Cstab(\k)\leq \Cpstab \k^\bstab.
\end{equation}
This polynomial growth condition on the stability constant is certainly not satisfied in general; see \cite{betcke} for the example of a so-called trapping domain that exhibits at least an exponential growth of the norm of the solution operator with respect to the wave number. Hence, the assumption~\eqref{e:stabpoly} puts implicit conditions on the domain $\Omega$ and the configuration of the boundary components. Sufficient geometric conditions in two or three dimensions that ensure \eqref{e:stabpoly} with $\bstab=0$ can be found in the original work of Melenk \cite{melenk_phd} (which based on the choice of a particular test function previously used in \cite{makridakis}) and its generalizations \cite{hetmaniukphd,feng,hetmaniuk,MelenkEsterhazy}. Among the known admissible setups are the case of a Robin boundary condition ($\Gamma_R=\partial\Omega$) on a Lipschitz domain $\Omega$ \cite{MelenkEsterhazy}. Another example is the scattering of acoustic waves at a sound-soft scatterer occupying the star-shaped polygonal or polyhedral domain $\Omega_D$ where the Sommerfeld radiation condition is approximated by the Robin boundary condition on the boundary of some artificial convex polygonal or polyhedral domain $\Omega_R\supset\bar{\Omega}_D$; see \cite{hetmaniuk}.

Given some linear functional $g$ on $V$, the adjoint problem of \eqref{e:modelvar} seeks $z\in V$ such that, for any $v\in V$, 
\begin{equation}\label{e:modelvara}
 a(v,z) = (v,g)_\Omega.
\end{equation}
Note that the adjoint problem is itself a Helmholtz problem in the sense that $S^*(g)=\overline{S(\bar{f})}$, where $S$ is the solution operator of  \eqref{e:modelvar} and $S^*$ is the solution operator of the adjoint problem \eqref{e:modelvara} \cite[Lemma 3.1]{MS11}. Hence, \eqref{e:modelvara} enjoys the same stability properties as \eqref{e:modelvar}.   

According to \cite{MelenkEsterhazy}, the stability \eqref{e:stab} for $f\in L^2(\Omega)$ implies well-posedness for all bounded linear functionals $f$ on $V$.
\begin{lemma}[well-posedness]\label{l:well}
The sesquilinear form $a$ of \eqref{e:a} satisfies
\begin{align}\label{e:infsup}
\inf_{u\in V\setminus\{0\}}\sup_{v\in V\setminus\{0\}}\frac{\Re a(u,v)}{\|u\|_V\|v\|_V}\geq \frac{1}{2\Cstab(\k)\k}.
\end{align}
Furthermore, for every $f\in V'$ (the space of bounded antilinear functionals on $V$) the problem \eqref{e:model} is
uniquely solvable, and its solution $u\in V$ satisfies the a priori bound
\begin{equation}\label{e:stab2}
\|u\|_V\leq \Cstab(\k)\k\|f\|_{V'}.
\end{equation}
\end{lemma}
Under the additional assumption \eqref{e:stabpoly} that the stability constant grows at most polynomially in $\k$, Lemma~\ref{l:well} shows polynomial well-posedness in the sense of \cite{MelenkEsterhazy}, i.e., polynomial-in-$\k$-bounds for the norm of the solution operator. 
\begin{proof}[Proof of Lemma~\ref{l:well}]
The proof of \eqref{e:infsup} is almost verbatim the same as that of \cite[Theorem 2.5]{MelenkEsterhazy} which covers the particular case $\Gamma_R=\partial\Omega$ and relies on a standard argument for sesquilinear forms satisfying a
G{\aa}rding inequality. Given $u\in V$, define $z\in V$ as the solution of 
\begin{equation*}
2\k^2(v,u)_\Omega = a(v,z),\quad\text{for all }v\in V.
\end{equation*}
The stability \eqref{e:stab} implies that 
\begin{equation}\label{e:infsup1}
\|z\|_V\leq2\Cstab(\k)\k^2\|u\|_\Omega\leq2\Cstab(\k)\k\|u\|_V.
\end{equation}
Set $v = u+z$ and observe that 
\begin{equation}\label{e:infsup2}
\Re a(u,v)= \|u\|_V^2.
\end{equation}
The combination of \eqref{e:infsup1} and \eqref{e:infsup2} yields \eqref{e:infsup}. Note that an analogue inf-sup condition can be proved for the adjoint of the bilinear form $a$ so that the Banach-Ne{\v{c}}as-Babu{\v{s}}ka theorem yields the unique solvability of both the primal and the adjoint problem as well as the a priori estimate \eqref{e:stab2}. 
\end{proof}

\section{Standard finite elements and quasi-interpolation}\label{s:pre}
This section recalls briefly the notions of simplicial finite element meshes and patches, standard finite element spaces and corresponding quasi-interpolation operators. In this paper, we will focus on linear finite elements based on triangles or tetrahedrons but higher order elements based and other types of meshes or even mesh-free approaches would be possible as well. The key property that we will exploit in the construction of the method is a partition of unity property of the basis; see \cite{HMP14}.
\subsection{Meshes, patches and spaces}\label{ss:mesh}
Let $\tri_H$ denote some regular (in the sense of \cite{CiarletPb}) simplicial finite element mesh of $\Omega$ with mesh size $H$. The mesh size is denoted by $H$ because, later on, there will be a second smaller discretization scale $h<H$. The mesh size $H$ may vary in space. In this case $H$ is function representing the local mesh width on the element level in the usual way. If no confusion seems likely, we will use still use $H$ to denote the maximal mesh size. As usual, the error analysis depends on some constant $\gamma>0$ that reflects the shape regularity of the finite element mesh $\triH$. 

The first-order conforming finite element space with respect to the mesh $\tri_H$ is given by
\begin{equation}\label{e:couranta}
V_H:=\{v\in V\;\vert\;\forall T\in\tri_H,v\vert_T \text{ is a polynomial of total degree}\leq 1\}.
\end{equation}
Let $\N_H$ denote the set of all vertices of $\tri_H$ that are not elements of the Dirichlet boundary. Every vertex $z\in\N_H$ represents a degree of freedom via the corresponding real-valued nodal basis function $\phi_z\in V_H$ determined by nodal values
\begin{equation*}
 \phi_z(z)=1\;\text{ and }\;\phi_z(y)=0\quad\text{for all }y\neq z\in\N_H.
\end{equation*}
The $\phi_z$ form a basis of $V_H$. 



The construction of the method and its analysis frequently uses the concept of coarse finite element patches. Such patches are agglomerations of elements of $\tri_H$. More precisely, we define patches $\Omega_{T,\ell}$ of variable order $\ell\in\mathbb{N}$ about an element $T\in\tri_H$ by
\begin{equation}\label{e:omegaT}
 \left\{\begin{array}{l}
 \Omega_{T,1}:= \cup\{T'\in\tri_H\;\vert\;T'\cap T\neq\emptyset\},\vspace{.5ex}\\
 \Omega_{T,\ell}:=\cup\left\{T'\in\tri_H\;|\;T'\cap {\overline\Omega}_{T,{\ell-1}}\neq\emptyset\right\},\; \ell=2,3,4\ldots .
\end{array}\right.
\end{equation}
In other words, $\Omega_{T,1}$ equals the union of $T$ and its neighbors and $\Omega_{t,\ell}$ is derived from $\Omega_{T,\ell-1}$ by adding one more layer of neighbors. 

Note that, for a fixed $\ell\in\mathbb{N}$, the element patches have finite overlap in the following sense. There exists a constant $\Coll{\ell}>0$ such that 
\begin{equation}\label{e:Col}
 \max_{T\in\tri_H}\#\{K\in\triH\;\vert\;K\subset \Omega_{T,\ell}\}\leq\Coll{\ell}. 
\end{equation}
The constant $\Col:=\Coll{1}$ equals the maximal number of neighbors of an element plus itself and there exists some generic constant $\Col'$ such that, for any $\ell>1$,  
\begin{equation*}
\Coll{\ell}\leq \max\left\{\#\tri_H, \Col' \ell^d \|H\|_{L^\infty(\Omega_{T,\ell})}\|H^{-1}\|_{L^\infty(\Omega_{T,\ell})}\right\}.
\end{equation*}

\subsection{Quasi-interpolation}\label{ss:quasi}
A key tool in the design and the analysis of the method is some bounded linear surjective Cl\'ement-type (quasi-)interpolation operator $\IH: V\rightarrow V_H$ as it is used in the a posteriori error analysis of finite element methods \cite{MR1706735}. Given $v\in V$, $\IH v := \sum_{z\in\N_H}\alpha_z(v)\phi_z$ defines a (weighted) Cl\'ement interpolant with nodal functionals
\begin{equation}\label{e:clement}
 \alpha_z(v):=\frac{(v, \phi_z)_\Omega}{(1,\phi_z)_\Omega}
\end{equation}
for $z\in\N_H$ and the hat functions $\phi_z$. Recall the (local) approximation and stability properties of the interpolation operator $\IH$ \cite{MR1706735}. There exists a generic constant $\Cint$ such that, for all $v\in V$ and for all $T\in\triH$,
\begin{equation}\label{e:interr}
H_T^{-1}\|v-\IH v\|_{L^{2}(T)}+\|\nabla(v-\IH v)\|_{L^{2}(T)}\leq \Cint \| \nabla v\|_{L^2(\Omega_{T,1})}.
\end{equation}
The constant $\Cint$ depends on the shape regularity parameter $\gamma$ of the finite element mesh $\triH$ but not on the local mesh size $H_T$. 

Note that the space $V_H$ is invariant under $\IH$ but $\IH$ is not a projection, i.e., $\IH v_H\neq v_H$ for $v_H\in V_H$ in general. However, since $\IH\vert_{V_H}$ can be interpreted as a diagonally scaled mass matrix, $\IH$ is invertible on the finite element space $V_H$ and the concatenation $(\IH\vert_{V_H})^{-1}\circ\IH:V\rightarrow V_H$ is a projection. For our particular choice of interpolation operator, one easily verifies that $(\IH\vert_{V_H})^{-1}\circ\IH$ equals the $L^2$-orthogonal projection $\Pi_H:V\rightarrow V_H$ onto the finite element space; see also \cite[Remark 3.1]{MP12}.
Recall that $\Pi_H$ is also stable in $V$, 
\begin{equation}\label{e:CPH}
 \|\Pi_H v\|_V\leq \CPH \|v\|_V\quad\text{for all }v\in V,
\end{equation}
where $\CPH$ depends only on the parameter $\gamma$ if the grading of the mesh is not too strong \cite{yserentant}.

While $\IH\vert_{V_H}$ is a local operator (a sparse matrix) its inverse $(\IH\vert_{V_H})^{-1}$ is not. However, there exists some bounded right inverse $\Iloc:V_H\rightarrow V$ of $\IH$ that is local. More precisely, there exists some generic constant $\Cint'$ depending only on $\gamma$ such that, for all $v_H \in V_H$,
\begin{align}\label{e:inv}
\left\{\;\begin{array}{rcl}\IH(\Iloc v_H)&=&v_H,\vspace{.5ex}\\
\|\nabla \Iloc v_H \|_\Omega&\leq&\Cint' \| \nabla v_H \|_\Omega,\vspace{.5ex}\\
\support(\Iloc v_H)&\subset& \bigcup\{T\;\vert\;T\in\tri_H:\;T\cap \overline{\support(v_H)}\neq\emptyset\}.
\end{array}\right.
 \end{align}
The third condition simply means that the support of $\Iloc v_H$ must not exceed the support of $v_H$ plus one layer of coarse elements. Note that $\Iloc v_H$ is not a finite element function on $\triH$ in general. An explicit construction of $\Iloc$ and a proof of the properties \eqref{e:inv} can be found in \cite[Lemma 1]{HMP14}. 

\begin{remark}[Other quasi-interpolation operators] 
We shall emphasize that the choice of a quasi-interpolation operator is by no means unique and a different choice might lead to a different multiscale method. A choice that turned out to be useful in previous works \cite{Brown.Peterseim:2014,Peterseim.Scheichl:2014} is the following one. Given $v\in V$, $\mathcal{Q}_H v := \sum_{z\in\N_H}\alpha_z(v)\phi_z$ defines a Cl\'ement-type interpolant with nodal functionals
\begin{equation}\label{e:QH}
 \alpha_z(v):=\left(\Pi_{H,\Omega_z}v\right)(z)
\end{equation}
for $z\in\N_H$. Here, $\Pi_{H,\Omega_z}v$ denotes the $L^2$-orthogonal projection of $v$ onto standard $P_1$ finite elements on the nodal patch $\Omega_z:=\support\phi_z$ and $\alpha_z(v)$ is the evaluation of this projection at the vertex $z$. 
We will show in the numerical experiment of Section~\ref{s:numexp} that the choice of the interpolation can affect the practical performance of the method significantly.
\end{remark}

\section{Global wave number adapted approximation}\label{s:omd}
This section introduces new (non-polynomial) approximation spaces  for the model Helmholtz problem under consideration. The spaces are mesh-based in the sense that degrees of freedom (or basis functions) are associated with vertices. The support of the basis functions is not local in general but quasi-local in the sense of some very fast decay of their moduli. Their replacement by localized computable basis functions in practical computations is possible; see Sections~\ref{s:loc} and \ref{s:fine}.
   
The ideal method requires the following assumption on the numerical resolution. 
\begin{assumption}[resolution condition]\label{a:resolution}
Given the wave number $\k$ and the constants $\Cint$ from \eqref{e:interr} and $\Col$ from \eqref{e:Col}, we assume that the mesh width $H$ satisfies 
\begin{equation}\label{e:resolution}
 H\k\leq \frac{1}{\sqrt{2\Col}\Cint}.
\end{equation}
\end{assumption}
Note that this assumption is quasi-minimal in the sense that a certain number of degrees of freedom per wave length is a necessary condition for the meaningful approximation of highly oscillatory waves.

\subsection{An ideal method}
The derivation of the method follows general principles of variational multiscale methods; cf. \cite{Hughes:1995,MR1660141,MR2300286} and \cite{Malqvist:2011}. 
Our construction of the approximation space starts with the observation that the space $V$ can be decomposed into 
the finite element space $V_H$ and the remainder space
\begin{equation}\label{e:RH}
 W:=\kernel \IH.
\end{equation}
The particular choice of $\IH$ implies that the  decomposition 
\begin{equation}\label{e:decompL2}
 V=V_H\oplus W
\end{equation}
is orthogonal in $L^2(\Omega)$ and, hence, stable.
We shall say that this $L^2$-orthogonality will not be crucial in this paper and that any choice of $\IH$ that allows a stable splitting of $V$ into its image and its kernel is possible, for instance $\mathcal{Q}_H$ defined in \eqref{e:QH}.

The subscale corrector $\Cor_\infty$ is a linear operator that maps $V$ onto $W$. Given $v\in V$, define the corrector $\Cor_\infty v\in W$ as the unique solution (cf. Lemma \ref{l:wellcor} below) of the variational problem
\begin{equation}\label{e:corp}
 a(\Cor_\infty v,w) = a(v,w),\quad\text{for all }w\in W. 
\end{equation} 
The subscript notation $\infty$ is consistent with later modifications $\Cor_\ell$ of the corrector, where the computation is restricted to local subdomains of size $\ell H$. It also reflects the infeasibility of the ideal method discussed in this section.
  
To deal with the lack of hermitivity, we will use the adjoint corrector $\Cor_\infty^* v\in W$ that solves the adjoint variational problem
\begin{equation}\label{e:cord}
 a(w,\Cor_\infty^* v) = a(w,v),\quad\text{for all }w\in W. 
\end{equation}
It turns out that
\begin{equation}
\Cor_\infty^* v = \overline{\Cor_\infty\bar{v}}
\end{equation}
holds for the model problem under consideration. 
Under Assumption~\ref{a:resolution}, the corrector problems \eqref{e:corp} and \eqref{e:cord} are well-posed.
\begin{lemma}[well-posedness of the correction operator]\label{l:wellcor}
The resolution condition of Assumption~\ref{a:resolution} implies that
$\|\nabla\cdot\|_\Omega$ and $\|\cdot\|_V$ are equivalent norms on $W$,
 \begin{equation}\label{e:equiv}
 \|\nabla w\|_\Omega \leq \|w\|_V\leq  \sqrt{\tfrac{3}{2}} \|\nabla w\|_\Omega, \quad\text{for all }w\in W,
\end{equation}
the sesquilinear form $a$ is $W$-elliptic,
\begin{equation}\label{e:ell}
 \Re a(w,w)\geq \tfrac{1}{3} \|w\|^2_V, \quad\text{for all }w\in W,
\end{equation}
and the correction operators $\Cor_\infty$, $\Cor_\infty^*$ are well-defined and stable,
\begin{equation}\label{e:Corstab}
 \|\Cor_\infty v\|_V=\|\Cor_\infty^* v\|_V\leq \CC\|v\|_V, \quad\text{for all }v\in V,
\end{equation}
where $\CC:=3\Ca$ with $\Ca$ from \eqref{e:acont}.
\end{lemma}
\begin{proof}
For any $w\in W$, the property $\IH w=0$, the approximation property \eqref{e:interr} of the quasi-interpolation operator, the bounded overlap of element patches $\Col$ and \eqref{e:resolution} yield
\begin{align*}
 \k^2(w,w)_\Omega &= \k^2(w-\IH w,w-\IH w)_\Omega\\
 &\leq \Col\Cint^2\k^2H^2\|\nabla w\|_\Omega^2\\
 &\leq \tfrac{1}{2}\|\nabla w\|_\Omega^2.
\end{align*}
This implies \eqref{e:equiv} and \eqref{e:ell}. Since the sesquilinear form $a$ is bounded \eqref{e:acont}, the well-posedness of \eqref{e:corp} and \eqref{e:cord} and the stability estimate \eqref{e:Corstab} follow from the Lax-Milgram theorem. 
\end{proof}

Since non-trivial projections on Hilbert spaces have the same operator norm as their complementary projections (see \cite{MR2279449} for a proof), the continuity of the projection operators $\Cor_\infty,\Cor_\infty^*$ implies the continuity of their complementary projections $(1-\Cor_\infty),(1-\Cor_\infty^*):V\rightarrow V$, that is, 
\begin{equation}\label{e:L2stab}
 \|(1-\Cor_\infty) v\|_V= \|(1-\Cor_\infty^*) v\|_V\leq \CC\|v\|_V, \quad\text{for all }v\in V,
\end{equation}
where $\CC=3\Ca$ is the constant from \eqref{e:Corstab}

The image of the finite element space $V_H$ under $(1-\Cor_\infty)$, 
\begin{equation}\label{e:coarse}
 \Vp:=(1-\Cor_\infty)V_H,
\end{equation}
defines a modified discrete approximation space. The space $\Vp$ will be the prototypical trial space in our method. The corresponding test space is
\begin{equation}\label{e:coarsed}
 \Vd:=(1-\Cor_\infty^*)V_H.
\end{equation}
Note that $W$ equals the kernel of both operators, $(1-\Cor_\infty)$ and $(1-\Cor_\infty^*)$. This implies that $\Vp$ is the image of $(1-\Cor_\infty)$ and $\Vd$ is the image of $(1-\Cor_\infty^*)$. The key properties of the spaces $\Vp$ and $\Vd$ are given in the subsequent lemma.  
\begin{lemma}[primal and dual decomposition]\label{l:aod} 
If the resolution condition of Assumption~\ref{a:resolution} is satisfied, then the decompositions
$$V = \Vp\oplus W = \Vd\oplus W$$ are stable. More precisely, any function $v\in V$ can be decomposed uniquely into 
$$v=\vp+w\quad\text{ and }\quad v=z_{H,\infty}+\overline{w}$$
and $$\max\{\|\vp\|_V,\|z_{H,\infty}\|_V,\|w\|_V\}\leq\CC\|v\|_V,$$
where $\vp:=(1-\Cor_\infty)v\in\Vp$, $z_{H,\infty}:=(1-\Cor_\infty^*)v\in\Vd$ and $w:=\Cor_\infty v\in W$.

The decompositions satisfy the following relations:
For any $\vp\in \Vp$ and any $w\in W$, it holds that
\begin{equation}\label{e:orthoa}
 a(\vp,w)=0.
\end{equation}
For any $z_{H,\infty}\in \Vd$ and any $w\in W$, it holds that
\begin{equation}\label{e:orthoad}
 a(w,z_{H,\infty})=0,
\end{equation}
\end{lemma}
\begin{proof}
The results readily follow from the construction of $\Cor_\infty$ and $\Cor_\infty^*$. 
\end{proof}

The Petrov-Galerkin method for the approximation of \eqref{e:modelvar} based on the trial-test pairing $(\Vp,\Vd)$ seeks $\up\in\Vp$ such that, for all $\vd\in\Vd$,
\begin{equation}\label{e:Galerkinglobal}
 a(\up,\vd)=(f,\vd)_\Omega.
\end{equation}
We shall emphasize that we do not consider the method \eqref{e:Galerkinglobal} for actual computations because the natural bases of the trial (resp. test) space, i.e., the image of the standard nodal basis of the finite element space under the operator $1-\Cor_\infty$ (resp. $1-\Cor_\infty^*$) is not sparse (or local) in the sense that the basis function $(1-\Cor_\infty)\phi_z$ (resp. $(1-\Cor_\infty^*)\phi_z$) have global support in $\Omega$ in general. Moreover the corrector problems are infinite dimensional problems. We will, hence, refer to the method \eqref{e:Galerkinglobal} as the ideal or global method. Later on, we will show that there are feasible nearby spaces with a sparse basis based on localized corrector problems (cf. Theorem~\ref{t:errorlocalization} below). We will also discretize these  localized corrector problems and analyze this perturbation in Section~\ref{s:fine}. 

\subsection{Stability and accuracy of the ideal method}
The ideal method admits a unique solution and is stable and accurate independent of $\k$ as long as the resolution condition $H\k\lesssim1$ is satisfied. The ``orthogonality'' relation \eqref{e:orthoa} induces stability.
\begin{theorem}[stability]\label{t:stabglob}
Let Assumption~\ref{a:resolution} be satisfied. Then the trial space $\Vp$ and test space $\Vd$ satisfy the discrete inf-sup condition
\begin{equation}\label{e:infsupdisc}
\inf_{\up\in \Vp\setminus\{0\}}\sup_{\vd\in \Vd\setminus\{0\}}\frac{\Re a(\up,\vd)}{\|\up\|_V\|\vd\|_V}\geq \frac{1}{2\CC\Cstab(\k)\k}.
\end{equation}
\end{theorem}
\begin{proof}
Observe that $(1-\Cor_\infty^*):V\rightarrow \Vd$ is a Fortin operator (as in the theory of mixed methods \cite{Fortin}), i.e., a bounded linear operator that satisfies
\begin{align*}
a(\up,(1-\Cor_\infty^*) v)&=a(\up, v)-\underbrace{a(\up,\Cor_\infty^* v)}_{=0; \text{see \eqref{e:orthoa}}}=a(\up,v),
\end{align*}
for all $\up\in \Vp$ and any $v\in V$.
Hence, the assertion follows from the inf-sup condition \eqref{e:infsup} on the continuous level and the continuity of $1-\Cor_\infty^*$ \eqref{e:Corstab}, 
\begin{align*}
\inf_{\up\in \Vp\setminus\{0\}}&\sup_{\vd\in \Vd\setminus\{0\}}\frac{\Re a(\up,\vd)}{\|\up\|_V\|\vd\|_V}\\&= \inf_{\up\in \Vp\setminus\{0\}}\sup_{v\in V\setminus\{0\}}\frac{\Re a(\up,(1-\Cor^*) v)}{\|\up\|_V\|(1-\Cor^*) v\|_V}\\
&\geq \frac{1}{\CC}\inf_{u\in V\setminus\{0\}}\sup_{v\in V\setminus\{0\}}\frac{\Re a(u,v)}{\|u\|_V\| v\|_V}\\
&\geq \frac{1}{2\CC\Cstab(\k)\k}.
\end{align*}
\end{proof}
The error estimate follows from the above discrete inf-sup condition, the ``orthogonality'' relation \eqref{e:orthoad}, and the Lax-Milgram theorem.
\begin{theorem}[error of the ideal method]
Let $u\in V$ solve \eqref{e:modelvar}. If the resolution condition of Assumption~\ref{a:resolution} is satisfied, then $\up=(1-\Cor_\infty) u\in \Vp$ is the unique solution of \eqref{e:Galerkinglobal}, that is, the Petrov-Galerkin approximation of $u$ in the subspace $\Vp$ with respect to the test space $\Vd$.
Moreover, it holds that 
\begin{equation}\label{e:errorGalerkin}
 \|u-\up\|_V\leq 3\sqrt{\Col}\Cint\|H f\|_\Omega.
\end{equation}
\end{theorem}
\begin{proof}
The Galerkin property \eqref{e:Galerkinglobal} of $\up=(1-\Cor_\infty)u$ follows from $\eqref{e:orthoad}$. Hence, the error $u-\up=\Cor_\infty u\in W$ satisfies
\begin{equation*}
 a(\Cor_\infty u,\Cor_\infty u)=a(u,\Cor_\infty u)=(f,\Cor_\infty u)_\Omega.
\end{equation*}
Since the sesquilinear form $a$ is $W$-elliptic (cf. \eqref{e:ell}), this yields the error estimate
\begin{equation*}
 \|u-\up\|_V^2\leq 3 \left|(f,\Cor_\infty u)_\Omega\right|.
\end{equation*}
Since $\IH\Cor_\infty u=0$, Cauchy inequalities and the interpolation error estimate \eqref{e:interr} readily yield the assertion.
\end{proof}
\begin{remark}[quasi-optimality]\label{r:quasioptimal} We shall say that the ideal method is also quasi-optimal in the following sense
 \begin{equation}\label{e:quasioptimal}
  \|u-\up\|_V\leq 3\Ca\inf_{\vp\in\Vp}\|u-\vp\|_V.
 \end{equation}
Moreover, since $\Pi_H\Cor_\infty u=0$, it holds that $\Pi_H u = \Pi_H \up$. This means that the ideal method provides the $L^2$-best approximation in the standard finite element space $V_H$,
 \begin{equation}\label{e:quasioptimal2}
  \|u-\Pi_H \up\|_\Omega=\min_{v_H\in V_H}\|u-v_H\|_\Omega.
 \end{equation}
Since $\IH(u-\up)=0$, $\up$ also satisfies the $L^2$ bound
 \begin{equation}\label{e:quasioptimal3}
  \|u-\up\|_\Omega\leq \sqrt{\Col}\Cint H \|u-\up\|_V.
 \end{equation}
\end{remark}

\begin{remark}[further stable variants of the method]\label{e:galerkin} We shall also mention at this point that the ``orthogonality'' relations \eqref{e:orthoa} and \eqref{e:orthoad} imply that, for any $u_H,v_H\in V_H$,
\begin{multline*}
 a((1-\Cor_\infty)u_H,v_H)=a((1-\Cor_\infty)u_H,(1-\Cor_\infty)v_H)\\=a((1-\Cor_\infty)u_H,(1-\Cor_\infty^*)v_H)\\=a(u_H,(1-\Cor_\infty^*)v_H)=a((1-\Cor_\infty^*)u_H,(1-\Cor_\infty^*)v_H).
\end{multline*}
This means that the Galerkin methods in $\Vp$ or $\Vd$ as well as Petrov-Galerkin methods based on the pairings $(\Vp,V_H)$ or $(V_H,\Vd)$ lead to stable and accurate discretizations. The latter Petrov-Galerkin method based on $(V_H,\Vd)$ is closely related to a variational multiscale stabilization of the standard $P_1$ finite element method and seeks $u_H\in V_H$ such that, for all $v_H\in V_H$,
\begin{equation}\label{e:Galerkinglobalstab}
 a(u_H,\vd)-a(u_H,\Cor_\infty^*v_H)=(f,v_H)_\Omega-(f,\Cor_\infty^*v_H)_\Omega.
\end{equation}
This stabilized method will be studied experimentally in Section~\ref{s:numexp}. 
\end{remark}

\subsection{Exponential decay of element correctors}\label{ss:loc}
Given some finite element function $v\in V$, its correction $\Cor_\infty v_H$ can be composed by element correctors $\Cor_{T,\infty}$, $T\in\tri_H$ in the following way:
\begin{equation}\label{e:corelem}
 \Cor_\infty v = \sum_{T\in\tri_H}\Cor_{T,\infty}(v\vert_T),
\end{equation}
where $\Cor_{T,\infty}(v\vert_T)\in W$ solves
\begin{equation}\label{e:corelem2}
 a(\Cor_{T,\infty}(v\vert_T),w)=a_T(v,w):=\int_T\nabla v\cdot\nabla \bar{w}\dx - \k^2\int_T\hspace{-1ex} v\bar{w}\dx- i\k\int_{\partial T\cap\Gamma_R}\hspace{-2ex}v\bar{w}\ds,
\end{equation}
for all $w\in W$.
Dual corrections can be split into element contributions in an analogue way,
\begin{equation}\label{e:corelem1d}
 \Cor^* v = \sum_{T\in\tri_H}\Cor^*_T(v\vert_T),
\end{equation}
where $\Cor^*_{T,\infty}(v\vert_T):=\overline{\Cor_{T,\infty}(\bar{v}\vert_T)}\in W$.  

The well-posedness of the element correctors is a consequence of Lemma~\ref{l:wellcor}. Moreover, it holds that
\begin{equation}\label{e:CorstabT}
 \|\Cor_{T,\infty} v\|_V=\|\Cor^*_{T,\infty} v\|_V\leq \CC\|v\|_{V(T)}, \quad\text{for all }v\in V,
\end{equation}
where $V(T)$ denotes the restriction of the space $V$ to the element $T$, and $\|v\|_{V(T)}^2:=\k^2\|v\|_{L^2(T)}^2+\|\nabla v\|_{L^2(T)}^2$. 

The major observation is that the moduli of the element correctors $\Cor_T v$ and $\Cor^*_T v$ decay very fast outside $T$.
\begin{theorem}[exponential decay of element correctors]\label{t:exp}
If the resolution condition of Assumption~\ref{a:resolution} is satisfied, then there exist constants $\Cdec>0$ and $\beta<1$ independent of $H$ and $\k$ such that for all $v\in V$ and all $T\in\tri_H$ and all $\ell\in\mathbb{N}$, the element correctors $\Cor_{T,\infty} v$ satisfy 
\begin{equation}\label{e:decay}
\|\nabla \Cor_{T,\infty} v\|_{\Omega\setminus\Omega_{T,\ell}} \leq \Cdec \beta^\ell \|\nabla v\|_{T}.
\end{equation}
The constants satisfy $\beta\leq\sqrt[14]{C_1/(C_1+\tfrac{1}{2})}<1$ and $\Cdec\leq\sqrt{\CC(C_1+\tfrac{1}{2})/C_1}$ for some $C_1>0$ (defined in the proof below) that depends only on the shape regularity parameter $\gamma$ of the mesh $\tri_H$. 
\end{theorem}
\begin{remark}[Rate of decay]
According to practical experience, the bound on the decay rate $\beta$ seems to be rather pessimistic. The rates observed in numerical experiments were between $1/3$ and $2/3$. 
\end{remark}
\begin{proof}[Proof of Theorem~\ref{t:exp}]
Let $T\in\tri_H$ be arbitrary but fixed and let $\ell\in \mathbb{N}$ with $\ell\geq 7$ and let the element patches $\Omega_{T,\ell},\Omega_{T,\ell-1},\ldots,\Omega_{T,\ell-7}$ be defined as in \eqref{e:omegaT}. Set $\psi:=\Cor_{T,\infty} v$. 

We define the cut-off function $\eta$ (depending on $T$ and $\ell$) by 
\begin{equation*}
\eta(x) := \frac{\dist(x,\Omega_{T,\ell-4})}{\dist(x,\Omega_{T,\ell-4})+\dist(x,\Omega\setminus\Omega_{T,\ell-3})}
\end{equation*}
for $x\in\Omega$. 
Note that $\eta=0$ in the patch $\Omega_{T,\ell-4}$ and $\eta=1$ in $\Omega\setminus\Omega_{T,\ell-3}$. Moreover, $\eta$ is bounded between 0 and 1 and Lipschitz continuous with
\begin{equation}\label{e:cutoffH}
 \|H\nabla\eta\|_{L^{\infty}(\Omega)}\leq\gamma.
\end{equation} 
The choice of $\eta$ implies the estimates
\begin{align}
   \|\nabla \psi\|_{\Omega\setminus\Omega_{T,\ell-3}}^2 
   &= \Re\, (\nabla\psi,\nabla\psi)_{\Omega\setminus\Omega_{T,\ell-3}}\nonumber\\
   &\leq \Re\, (\nabla\psi,\eta\nabla\psi)_{\Omega}\nonumber\\
   &= \Re\, (\nabla\psi,\nabla\left(\eta\psi\right))_{\Omega} - \Re\, (\nabla\psi,\psi\nabla \eta)_{\Omega}\nonumber\\
   &\leq |\Re\, (\nabla\psi,\nabla\left(\eta\psi-\Iloc(\IH(\eta\psi))\right))_{\Omega}|\nonumber\\&\qquad +|\Re\,(\nabla\psi,\nabla\Iloc(\IH(\eta\psi)))_{\Omega}| + \left|\Re\, (\nabla\psi,\psi\nabla \eta)_{\Omega}\right|\nonumber\\
   & =: M_1+M_2+M_3.\label{e:decay1}
\end{align}
Note that the test function $\left(\eta\psi-\Iloc(\IH(\eta\psi))\right)\in W$ with support in $\Omega\setminus \Omega_{T,\ell-6}$. If $\ell\geq 6$, then $\eta\psi-\Iloc(\IH(\eta\psi))$ vanishes on $T$ and $a_T(v,\eta\psi-\Iloc(\IH(\eta\psi)))=0$. Hence, the definition \eqref{e:cord} of $\Cor_{T,\infty}$, the Cauchy-Schwarz inequality, the properties \eqref{e:interr} and \eqref{e:inv} of the interpolation operator $\IH$ and the resolution condition Assumption~\ref{a:resolution} imply
\begin{align}
M_1&:=|\Re\, (\nabla\psi,\nabla(\eta\psi-\Iloc(\IH(\eta\psi))))_{\Omega}|\nonumber\\&=\left|\k^2(\psi,\eta\psi-\Iloc(\IH(\eta\psi)))_{\Omega}\right|\nonumber\\
&\leq \Cint^2\Col(H\k)^2\|\nabla\psi\|_{\Omega\setminus\Omega_{T,\ell-6}}^2+\Cint^3\Cint'\Col(H\k)^2\|\nabla\psi\|_{\Omega_{T,\ell}\setminus\Omega_{T,\ell-7}}^2\nonumber\\
&\leq \tfrac{1}{2}\|\nabla \psi\|_{\Omega\setminus\Omega_{T,\ell}}^2 + \tfrac{1}{2}(1+\Cint\Cint')\|\nabla\psi\|_{\Omega_{T,\ell}\setminus\Omega_{T,\ell-7}}^2.\label{e:decay4}
\end{align}
Similar techniques and the Lipschitz bound \eqref{e:cutoffH} lead to upper bounds of the other terms on the right-hand side of \eqref{e:decay1},
\begin{align}
M_2&\leq\Cint'\Cint\|\nabla(\eta\psi)\|_{\Omega_{T,\ell-1}\setminus\Omega_{T,\ell-6}}\|\nabla\psi\|_{\Omega_{T,\ell-1}\setminus\Omega_{T,\ell-6}}\nonumber\\
&\leq \Cint'\Cint\left(\Cint\sqrt{\Col}\|H\nabla\eta\|_{L^\infty(\Omega)} +1\right)\|\nabla\psi\|_{\Omega_{T,\ell}\setminus\Omega_{T,\ell-7}}^2\label{e:decay3}
\end{align}
and
\begin{equation}\label{e:decay5}
M_3\leq \Cint\sqrt{\Col}\|H\nabla\eta\|_{L^\infty(\Omega)}\|\nabla\psi\|_{\Omega_{T,\ell-2}\setminus\Omega_{T,\ell-5}}^2.
\end{equation}
The combination of \eqref{e:decay1}--\eqref{e:decay5} readily yields the estimate
\begin{equation*}
   \tfrac{1}{2}\|\nabla \psi\|_{\Omega\setminus\Omega_{T,\ell}}^2 \leq C_1 \|\nabla \psi\|_{\Omega_{T,\ell}\setminus\Omega_{T,\ell-7}}^2,
\end{equation*}
where $C_1:=\tfrac{1}{2}+\tfrac{3}{2}\Cint\Cint'+(\Cint'\Cint+1)\Cint\sqrt{\Col}\gamma$ depends only on the shape regularity of the coarse mesh $\tri_H$.
Since 
\begin{equation*}
\|\nabla \psi\|_{\Omega_{T,\ell}\setminus\Omega_{T,\ell-7}}^2=\|\nabla \psi\|_{\Omega\setminus\Omega_{T,\ell-7}}^2 -\|\nabla \psi\|_{\Omega\setminus\Omega_{T,\ell}}^2,
\end{equation*}
this implies the contraction
\begin{equation*}
   \|\nabla \psi\|_{\Omega\setminus\Omega_{T,\ell}}^2 \leq \frac{C_1}{C_1+\tfrac{1}{2}} \|\nabla \psi\|_{\Omega\setminus\Omega_{T,\ell-7}}^2.
\end{equation*}
Hence,
\begin{equation*}
\|\nabla \psi\|_{\Omega\setminus\Omega_{T,\ell}}^2 \leq \left(\frac{C_1}{C_1+\tfrac{1}{2}}\right)^{\left\lfloor\tfrac{\ell}{7} \right\rfloor} \|\nabla \psi\|_{\Omega}^2\leq \CC \left(\frac{C_1}{C_1+\tfrac{1}{2}}\right)^{\left\lfloor\tfrac{\ell}{7} \right\rfloor} \|\nabla v\|_{T}^2,
\end{equation*}
and some algebraic manipulations yield the assertion.
\end{proof}

\section{Localized approximation}\label{s:loc}
This section localizes the corrector problems from $\Omega$ to patches $\Omega_{T,\ell}$; $\ell$ being a novel discretization parameter - the oversampling parameter.  
\subsection{Localized correctors}
The exponential decay of the element correctors (cf. Theorem~\ref{t:exp}) motivates their localized approximation on element patches. Given such a patch $\Omega_{T,\ell}$ for some $T\in\tri_H$ and $\ell\in\mathbb{N}$ define the localized remainder space 
\begin{equation}\label{e:RHloc}
W(\Omega_{T,\ell}):=\{w\in W\;\vert\;w\vert_{\Omega\setminus\Omega_{T,\ell}}=0\}
\end{equation}
and the localized sesquilinear form 
\begin{equation}\label{e:aloc}
 a_{D}(u,v):=(\nabla u,\nabla v)_{D}-\k^2(u,v)_{D} - i\k(u,v)_{\Gamma_R\cap\partial D}, 
\end{equation}
where $D$ is any subdomain of $\Omega$, e.g., $D=\Omega_{T,\ell}$ or $D=T$. Then, given some finite element function $v_H\in V_H$, its localized primal correction $\Cor_\ell v_H$ is defined via localized element correctors in the following way:
\begin{equation}\label{e:corelemp}
 \Cor_\ell v_H := \sum_{T\in\tri_H}\Cor_{T,\ell}(v_H\vert_T),
\end{equation}
where $\Cor_{T,\ell}(v_H\vert_T)\in W(\Omega_{T,\ell})$ satisfies
\begin{equation}\label{e:corelemp3}
 a_{\Omega_{T,\ell}}(\Cor_T(v_H\vert_T),w)=a_T(v_H,w) 
\end{equation}
for all $w\in W(\Omega_{T,\ell})$.
The localized dual correction is $\Cor^*_\ell v_H:=\overline{\Cor_\ell \bar{v}_H}$. It holds that  $\Cor^*_\ell=\Cor_\ell$ whenever $\Omega_{T,\ell}\cap\Gamma_R=\emptyset$. 

Note that \eqref{e:corelemp3} is truly localized insofar as the linear constraints $(w,\phi_z)_\Omega=0$ ($z\in\N_H$) that characterize an element $w\in W$ need to be checked only for $z\in\N_H\cap\Omega_{T,\ell}$ and are satisfied automatically for all other nodes if $w\in W(\Omega_{T,\ell})$. We shall also stress that if the mesh is (locally) structured so that some patches are equal up to translation or rotation with the same local triangulation, then also the corresponding correctors will coincide up to shift and rotation. This means that on a uniform mesh only $\mathcal{O}(1)$ interior cell problems need to be solved plus a number of cell problems that capture all possible intersections of the patches and the boundary parts. On polyhedral domains, this number depends only on the oversampling parameter $\ell$ and the number of boundary faces of the domain. 

Though being localized, the correctors $\Cor_{T,\ell}$ and $\Cor^*_{T,\ell}$ are still somewhat ideal because their evaluation requires the solution of an infinite-dimensional variational problem in the space $W(\Omega_{T,\ell})$. To be fully practical, we will also have to discretize the local corrector problems \eqref{e:corelemp3}. This step and the analysis of corresponding errors will be discussed Section~\ref{s:fine} below. 
The remaining part of the section is devoted to the analysis of the perturbation introduced by the localization of the correctors.

An error bound for the localized approximation of  the corrector $\Cor$ and its adjoint $\Cor^*$ is easily derived from the exponential decay property of Theorem~\ref{t:exp}. 
\begin{lemma}[local approximation of element correctors]\label{l:localization}If the resolution condition of Assumption~\ref{a:resolution} is satisfied, then, for any $T\in\tri_H$ and any $\ell\in\mathbb{N}$, it holds that
\begin{equation*}
 \|\nabla(\Cor_{T,\infty} v-\Cor_{T,\ell} v)\|_\Omega\leq \Cdec'\beta^{\ell} \|\nabla v\|_{T},
\end{equation*}
where $\beta<1$ is the constant from Theorem~\ref{t:exp} and \\$\textstyle\Cdec':=\left(6\Ca^2(1+\Cint^2\Cint'^2)\left(\tfrac{3}{2}+\Cint^2\Col\gamma^2\right)\right)^{1/2}\Cdec\beta^{-6}$.
\end{lemma}
\begin{proof}
Define the cut-off function $\eta$ (depending on $T$ and $\ell$) 
\begin{equation}\label{e:cutoff1}
\eta(x) := \frac{\dist(x,\Omega\setminus\Omega_{T,\ell-2})}{\dist(x,\Omega_{T,\ell-3})+\dist(x,\Omega\setminus\Omega_{T,\ell-2})}.
\end{equation}
Note that $\eta=1$ in $\Omega_{T,\ell-3}$ and $\eta=0$ outside $\Omega_{T,\ell-2}$. Moreover, $\eta$ is bounded between 0 and 1 and satisfies the Lipschitz bound \eqref{e:cutoffH}.
Since $\Cor_{T,\ell}v$ is the Galerkin approximation of $\Cor_{T,\infty}v$ and $\eta\Cor_{T,\infty}v-\Iloc(\IH(\eta\Cor_{T,\infty}v))\in W(\Omega_{T,\ell})$, C\'ea's lemma plus Lemma~\ref{l:wellcor}, the definition of $\Cor_{T,\infty}$ \eqref{e:corelem2}, the Cauchy-Schwarz inequality, the approximation property \eqref{e:interr} of the interpolation operator $\IH$, the shape regularity of the mesh (cf. \eqref{e:cutoffH}) and the resolution condition Assumption~\ref{a:resolution} imply 
\begin{align*}
 \|&\nabla(\Cor_{T,\infty} v-\Cor_{T,\ell} v)\|_\Omega^2\leq3\Ca^2\|\Cor_{T,\infty} v-(\eta\Cor_{T,\infty}v-\Iloc(\IH(\eta\Cor_{T,\infty}v))\|_V^2\\
 &\leq 6\Ca^2 \left(\|\nabla((1-\eta)\Cor_{T,\infty} v)\|_\Omega^2+\k^2\|(1-\eta)\Cor_{T,\infty} v\|_\Omega^2\right)\\
 &\quad +6\Ca^2\Cint^2\Cint'^2\left(\|\nabla(\eta\Cor_{T,\infty}v)\|_{\Omega_{T,\ell}\setminus\Omega_{T,\ell-5}}^2+\k^2\|\eta\Cor_{T,\infty}v\|_{\Omega_{T,\ell}\setminus\Omega_{T,\ell-5}}^2\right)\\
 &\leq 6\Ca^2(1+\Cint^2\Cint'^2) \left(\|\nabla\Cor_{T,\infty} v\|_{\Omega\setminus\Omega_{T,\ell-5}}^2\right.
\\
 &\quad \left.+\Cint^2\Col\|H\nabla\eta\|_{L^{\infty}(\Omega)}^2\|\nabla\Cor_{T,\infty} v\|_{\Omega\setminus\Omega_{T,\ell-6}}^2+\Cint^2\Col(H\k)^2\|\nabla\Cor_{T,\infty} v\|_{\Omega\setminus\Omega_{T,\ell-6}}^2\right)\\
 &\leq 6\Ca^2(1+\Cint^2\Cint'^2)\left(\tfrac{3}{2}+\Cint^2\Col\gamma^2\right)\|\nabla\Cor_{T,\infty} v\|_{\Omega\setminus\Omega_{T,\ell-6}}^2. 
 \end{align*}
This and Theorem~\ref{t:exp} readily imply the assertion. 
\end{proof}

\begin{theorem}[error of the localized corrections]\label{t:errorlocalization}
If the resolution condition of Assumption~\ref{a:resolution} is satisfied, then, for any $\ell\in\mathbb{N}$, it holds that
\begin{equation*}
 \|\nabla (\Cor_\infty v-\Cor_\ell v)\|_\Omega\leq \Cloc{\ell}\beta^{\ell} \|\nabla v\|_\Omega,
\end{equation*}
where $\Cloc{\ell}:=3\sqrt{3}\Coll{\ell+5}\Ca^{2} (1+\Cint'\Cint)(\tfrac{3}{2}+\Cint^2\Col\gamma^2)\Cdec'$.
\end{theorem}
\begin{proof}
Set $z:=\Cor_\infty v-\Cor_{\ell} v$ and, for any $T\in\tri_H$, set $z_T:=\Cor_{T,\infty} v-\Cor_{T,\ell} v$. The $W$-ellipticity of the sesquilinear form \eqref{e:ell} implies that
\begin{align}\label{e:overlap1}
 \tfrac{1}{3}\|\nabla z\|_\Omega^2&\leq \sum_{T\in\tri_H} a(z_T,z). 
\end{align}
Given some $T\in\tri_H$, let $\eta$ be the cutoff function defined by 
\begin{equation*}
\eta(x) := \frac{\dist(x,\Omega_{T,\ell+2})}{\dist(x,\Omega_{T,\ell+2})+\dist(x,\Omega\setminus\Omega_{T,\ell+3})},
\end{equation*}
that is $\eta=0$ in $\Omega_{T,\ell+2}$ and $\eta=1$ outside $\Omega_{T,\ell+3}$. Moreover, $\eta$ is bounded between 0 and 1 and satisfies the Lipschitz bound \eqref{e:cutoffH}.
Since $\operatorname{supp}\Iloc(\IH(\eta z))\subset \Omega\setminus\Omega_{T,\ell}$ and $\eta z - \Iloc(\IH(\eta z))\in W$, we have that 
\begin{equation*}
 a(z_T,\eta z - \Iloc(\IH(\eta z)))=a(\Cor_{T,\infty}v,\eta z - \Iloc(\IH(\eta z)))=0. 
\end{equation*}
Hence, 
\begin{equation*}
 a(z_T,z)=a(z_T,\Iloc(\IH(\eta z))+a(z_T,(1-\eta)z).
 \end{equation*}
The properties \eqref{e:interr} of the interpolation operator $\IH$ and the Lipschitz bound \eqref{e:cutoffH} lead to upper bounds 
\begin{equation}\label{e:overlap2}
 a(z_T,z)\leq \Ca (1+\Cint'\Cint)\sqrt{1+\Cint^2\Col\gamma^2}\|z\|_{V,\Omega_{T,\ell+5}}\|z_T\|_V.
\end{equation}
The combination of \eqref{e:overlap1} and \eqref{e:overlap2} plus a discrete Cauchy-Schwarz inequality and the bounded overlap \eqref{e:Col} of the element patches leads to 
 \begin{align}\label{e:overlap3}
 \|\nabla z\|_\Omega&\leq 2\Coll{\ell+3}\Ca (1+\Cint'\Cint)\sqrt{1+\Cint^2\Col\gamma^2}\left(\sum_{T\in\tri_H}\|z_T\|_V^2\right)^{1/2}.
\end{align}
This and Lemma~\ref{l:localization} readily yield the assertion.
\end{proof}

\subsection{Localized trial and test spaces}
The localized trial space $\Vpl\subset V$ is simply defined as the image of the classical finite element space $V_H$ under the operator $1-\Cor_\ell$,
\begin{equation}\label{e:coarsel}
 \Vpl:=(1-\Cor_\ell)V_H
\end{equation}
and the localized test space $\Vdl\subset V$ reads
\begin{equation}\label{e:coarseld}
 \Vdl:=(1-\Cor^*_\ell)V_H
\end{equation}

Note that both $\Vpl$ and $\Vdl$ are finite-dimensional with a local basis, 
\begin{equation*}
 \Vpl=\operatorname*{span}\{(1-\Cor_\ell)\phi_z\,\vert\,z\in \N_H\}\text{ and }\Vdl=\operatorname*{span}\{\overline{(1-\Cor_\ell)\phi_z}\,\vert\,z\in \N_H\},
\end{equation*}
where $\phi_z$ is the (real-valued) nodal basis of $V_H$ (cf. Section~\ref{ss:mesh}).

The Petrov-Galerkin method with respect to the trial space $\Vpl$ and the test space $\Vdl$ seeks $\upl\in\Vpl$ such that, for all $\vdl\in \Vdl$, 
\begin{equation}\label{e:VclGalerkin}
a(\upl,\vdl)=(f,\vdl)_\Omega.
\end{equation} 

\subsection{Stability of the localized method}
The stability of the localized methods requires the coupling of the oversampling parameter to the stability constant which we will now assume to be polynomial with respect to the wave number. 
\begin{assumption}[polynomial-in-$\k$-stability and logarithmic oversampling condition]\label{a:oversampling}
There are constants $\Cpstab>0$ and $\bstab \geq 0$ and a $\k_0>0$ that may depend on $\Omega$ and the partition of the boundary into $\Gamma_D$, $\Gamma_N$ and $\Gamma_R$ such that, for any $\k\geq\k_0$, the stability constant $\Cstab(\k)$ of \eqref{e:stab} satisfies \eqref{e:stabpoly},
\begin{equation*}
\Cstab(\k)\leq \Cpstab \k^\bstab.
\end{equation*}
Given the wave number $\k$ and the constants $\Cint$ from \eqref{e:interr} and $\Col$ from \eqref{e:Col}, we assume that the oversampling parameter $\ell$ satisfies 
\begin{equation}\label{e:oversampling}
 \ell\geq\frac{(\bstab+1)\log\k+\log\left(4\CC\Cpstab\CPH\sqrt{\tfrac{3}{2}}\Cloc{\ell}\Ca\right)}{|\log\beta|}.
\end{equation}
\end{assumption}
Since the constant $\Cloc{\ell}$ grows at most polynomially with $\ell$ (cf. \eqref{e:Col}), condition~\eqref{e:oversampling} is indeed satisfiable and the proper choice of $\ell$ will be dominated by the logarithm $\log\k$ of the wave number.

The stability of the localized method follows from the fact that the ideal pairing $(\Vp,\Vd)$ is stable and that $(\Vpl,\Vdl)$ is exponentially close. 
\begin{theorem}[stability of the localized method]\label{t:stabloc}
If the mesh width $H$ is sufficiently small in the sense of Assumption~\ref{a:resolution} ($H\k\lesssim 1$) and if the oversampling parameter $\ell\in\mathbb{N}$ is sufficiently large in the sense of Assumption~\ref{a:oversampling} ($\ell\gtrsim\log\k $), then the pairing of the localized spaces $\Vpl$ and $\Vdl$ satisfies the discrete inf-sup condition
\begin{equation}\label{e:infsupdiscl}
\inf_{\upl\in \Vpl\setminus\{0\}}\sup_{\vdl\in \Vdl\setminus\{0\}}\frac{\Re a(\upl,\vdl)}{\|\upl\|_V\|\vpl\|_V}\geq \frac{1}{4\CC\Cpstab\k^{\bstab+1}}.
\end{equation}
This ensures that, for any $f\in V'$, there exists a unique solution of the discrete problem \eqref{e:VclGalerkin}.
\end{theorem}
\begin{proof}
Let $\upl\in\Vpl$ and set $\up:=(1-\Cor)\Pi_H\upl$. Under the polynomial-in-$\k$ stability of Assumption~\ref{a:oversampling}, Theorem~\ref{t:stabglob} guarantees the existence of some $\vd\in\Vd$ with 
\begin{equation}\label{e:stabproof1}
 \Re a(\up,\vd)\geq \frac{1}{2\Cpstab\CC\k^{\bstab+1}}\|\up\|_V\|\vd\|_V.
\end{equation}
Set $\vdl:=(1-\Cor^*_\ell)\Pi_H\vp\in\Vdl$ and observe that \eqref{e:orthoad} yields 
\begin{equation*}
\begin{aligned}
 \Re a(\upl,\vdl)&=\Re a(\upl,\vdl-\vd)+a(\upl,\vd)\\
 &=\Re a(\upl,(\Cor^*-\Cor^*_\ell)\Pi_H\vdl)+\Re a(\up,\vd).
 \end{aligned}
 \end{equation*}
Hence,
\begin{equation}\label{e:stabproof2}
\begin{aligned}
 \Re a(\upl,\vdl)&\geq\Re  a(\up,\vd)-\Ca\|\upl\|_V\|(\Cor^*-\Cor^*_\ell)\Pi_H\vdl\|_V\\
 &\geq\Re  a(\up,\vd)-\CPH\sqrt{\tfrac{3}{2}}\Cloc{\ell}\Ca\beta^{\ell}\|\upl\|_V\|\vdl\|_V,
\end{aligned}
\end{equation}
where we have used \eqref{e:equiv}, Theorem~\ref{t:errorlocalization}, and \eqref{e:CPH}. This yields
\begin{equation}\label{e:stabproof3}
\begin{aligned}
 \Re a(\upl,\vdl)&\geq \frac{1}{\CC\Cstab\k^{\bstab+1}}\|\up\|_V\|\vd\|_V-C'\beta^{\ell}\|\upl\|_V\|\vdl\|_V\\
 &\geq \left(\frac{1}{2\CC\Cpstab\k^{\bstab+1}}-\CPH\sqrt{\tfrac{3}{2}}\Cloc{\ell}\Ca\beta^{\ell}\right)\|\upl\|_V\|\vdl\|_V,
\end{aligned}
\end{equation}
and Assumption~\ref{a:oversampling} readily implies the assertion. 
\end{proof}

\begin{theorem}[error of the localized method]\label{t:errorloc}
If the mesh width $H$ is sufficiently small in the sense of Assumption~\ref{a:resolution} ($H\k\lesssim 1$) and if the oversampling parameter $\ell\in\mathbb{N}$ is sufficiently large in the sense of Assumption~\ref{a:oversampling} ($\ell\gtrsim\log\k $), then the localized Petrov-Galerkin approximation $\upl\in\Vpl$ satisfies the error estimate
\begin{equation}\label{e:errorloc}
\|u-\upl\|_V\leq  6\sqrt{\Col}\Cint\|H f\|_\Omega+6\Ca\Cloc{\ell}\CPH\Cpstab \k^\bstab\beta^{\ell} \|f\|_\Omega.
\end{equation}
\end{theorem}
\begin{proof}
The proof is inspired by standard techniques for Galerkin methods (see \cite{schatz}, \cite[Thm. 5.7.6]{MR2373954}, \cite{sauterref}, \cite{banjaisauterref}). Set $\ep:=u-\upl$ and $\epl:=(1-\Cor_\ell)\Pi_H \ep\in\Vpl$. The triangle inequality yields
\begin{equation}\label{e:triangle}
 \|\ep\|_V\leq \|\ep-\epl\|_V+\|\epl\|_V.
\end{equation}

An Aubin-Nitsche duality argument shows that $\|\epl\|_V$ is controlled by some multiple of $\|\ep-\epl\|_V$. Let $\zdl\in \Vdl$ be the unique solution of the discrete adjoint variational problem
\begin{equation*}
 (\nabla\vpl,\nabla\epl)+\k^2(\vpl,\epl) = a(\vpl,\zdl),
\end{equation*}
for all $\vpl\in \Vpl$. Set $\zd:=(1-\Cor^*)\Pi_H\zdl$ and observe that 
\begin{align*}
 \|\epl\|_V^2 &= a(\epl,\zdl-\zd)+a(\epl,\zd)\\
 &= a(\epl,\zdl-\zd)+a(\ep,\zd)\\
 &= a(\epl,\zdl-\zd)+a(\ep,\zd-\zdl)\\
 &= a(\ep-\epl,(\Cor^*-\Cor^*_\ell)\Pi_H \zdl)\\
 &\leq \Ca\|\ep-\epl\|_V\|(\Cor^*-\Cor^*_\ell)\Pi_H \zdl\|_V.
\end{align*}
Under Assumption \eqref{e:stabpoly}, Theorem~\ref{t:errorlocalization}, Theorem~\ref{t:stabloc} and \eqref{e:CPH} readily yield
\begin{equation}\label{e:epl}
 \|\epl\|_V^2 \leq \Ca^2\Cloc{\ell}\beta^{\ell}\CPH\Cpstab\k^{\bstab+1}\|\ep-\epl\|_V^2.
\end{equation}
This, \eqref{e:triangle} and Assumption~\ref{a:oversampling} show that 
\begin{equation}\label{e:epl2}
 \|\ep\|_V \leq 
 2\|\ep-\epl\|_V.
\end{equation}

Since $\ep-\epl\in W$, the $W$-ellipticity \eqref{e:ell} yields 
\begin{equation}\label{e:term1}
 \|\ep-\epl\|_V^2\leq 3\Re a(\ep-\epl,\ep-\epl).
\end{equation}
The relation \eqref{e:orthoa} then yields 
\begin{multline}\label{e:term1b}
 a(\ep-\epl,\ep-\epl)=a(u,\ep-\epl)+a((\Cor-\Cor_\ell)\Pi_H u,\ep-\epl)\\
 \leq \left|(f,\ep-\epl)_\Omega\right|+\Ca\|(\Cor-\Cor_\ell)\Pi_H u\|_V\|\ep-\epl\|_V.
\end{multline}
This, Cauchy inequalities, interpolation error estimates \eqref{e:interr}, Theorem~\ref{t:errorlocalization} and the stability estimate \eqref{e:stabpoly} readily yield the bound 
\begin{equation}\label{e:eepl}
\|\ep-\epl\|_V \leq 3\sqrt{\Col}\Cint\|H f\|_\Omega+3\Ca\Cloc{\ell}\beta^{\ell}\CPH\Cstab \k^\bstab \|f\|_\Omega.
\end{equation}
The combination of \eqref{e:epl2} and \eqref{e:eepl} is the assertion.
\end{proof} 

\section{Fully discrete localized approximation}\label{s:fine}
As already mentioned before, the localized corrector problems \eqref{e:corelem2} are variational problems in infinite-dimensional spaces $W(\Omega_{T,\ell})$ that require further discretization. For the ease of presentation we restrict ourselves in this paper to the classical case of piecewise affine conforming elements on simplicial meshes but we emphasize that the technique easily transfers to more general situations and can be applied to a large variety of discretization schemes and, in particular, to $hp$ adaptive methods. 

So far, the presentation of the method was optimized with respect to theoretical aspects of the stability and error analysis. Here, we will present the method in  a slightly more practical fashion.
\subsection{The fully discrete method}
For any $T\in\tri_H$, choose an oversampling parameter $\ell=\ell_T$ (sufficiently large so that there is a chance that Assumption~\ref{a:oversampling} is satisfied). 
Let $\tri_h(\Omega_{T,\ell})$ be a regular (and possibly adaptive) mesh of width $h< H$ and consider the standard finite element space $V_h(\Omega_{T,\ell})$ of continuous piecewise polynomials of order $1$ (or any higher order) with respect to $\tri_h(\Omega_{T,\ell})$. 
Then the discretized local remainder space is defined by
\begin{equation*}
W_h((\Omega_{T,\ell}):=W(\Omega_{T,\ell})\cap V_h(\Omega_{T,\ell}).
\end{equation*}
For any vertex $y$ of $T$, compute the element corrector $\Cor_{T,\ell,h}\phi_y\in W_h((\Omega_{T,\ell})$ as the unique solution of the discrete cell problem 
\begin{equation*}
 a(\Cor_{T,\ell,h}\phi_y,w)=a_T(\phi_y,w),\quad\text{for all }w\in W_h((\Omega_{T,\ell}).
\end{equation*}
In practice, the linear constraints in the definition of  $W_h((\Omega_{T,\ell})$ related to the nodal functionals $\alpha_z$ from \eqref{e:clement} (for coarse nodes $z$ in the patch $\Omega_{T,\ell}$) are realized using Lagrangian multipliers so that the computation can be performed in the standard finite element space $V_h(\Omega_{T,\ell})$ (up to essential boundary conditions) and no explicit knowledge about a basis of $W_h((\Omega_{T,\ell})$ is required.

For every global vertex $z\in\N_H$, the corrector $\Cor_{\ell,h}\phi_z$ is then given by
\begin{equation*}
 \Cor_{\ell,h}\phi_z := \sum_{T\in\tri_H:z\text{ vertex of }T}\Cor_{T,\ell,h}\phi_z. 
\end{equation*}
This leads to modified basis functions $\tilde{\phi_z}:=\phi_z-\Cor_{\ell,h}\phi_z$ that span a discrete space
\begin{equation}
 V_{H,\ell,h}:=\operatorname*{span}\{\tilde{\phi_z}\;\vert\;z\in\N_H\}
\end{equation}
of the same dimension as the classical finite element space $V_H$. In this most general setting, the discretizations of the cell problems are completely independent and could be very different (e.g. high order polynomials on a coarse mesh for interior patches, or an adaptive discretization when corners of the physical domain are present in the patch).

The fully discrete localized Petrov-Galerkin method with respect to the trial space $\Vplh$ and the test space $\Vdlh$ seeks $\uplh\in\Vplh$ such that, for all $\vdlh\in \Vdlh$, 
\begin{equation}\label{e:VclhGalerkin}
a(\uplh,\vdlh)=(f,\vdlh)_\Omega.
\end{equation} 
\subsection{Error analysis of the fully discrete method}
An a priori error analysis of the general approach would follow the analysis of Section~\ref{s:loc} and trace the error of the additional perturbation depending on the local choice of the approximation space. However, this will require the estimation of the error $\Cor-\Cor_{\ell,h}$ or $\Cor_{\ell}-\Cor_{\ell,h}$, which appears to be non-trivial and requires, for instance, regularity results for the ideal correctors. This line will be followed in future research along with an a posteriori analysis of the method whereas we focus on a special case with a very simple argument in this paper. 

We restrict ourselves to the case of synchronized cell problems in the sense that there is an underlying global fine mesh $\tri_h$ that is a regular refinement of the coarse mesh $\tri_H$. In this case, the global fine space $V_h$ contains standard finite element functions and the spaces for the local cell problems are derived by restriction of $V_h$ to the patch. Then, the method in fact approximates $u_h$, where $u_h\in V_h$ is the Galerkin approximation in the global fine scale, that is,
\begin{equation}\label{e:modelh}
 a(u_h,v_h) = (f,v_h)_\Omega,\quad\text{for all }v_h\in V_h.
\end{equation} 
In the remaining part of this paper, we will refer to $u_h$ as the reference solution that we will compare our approximations with. It is clear that if $h$ is sufficiently small, then the problem \eqref{e:modelh} is well-posed. Since it is unknown in general how to quantify what ``sufficiently small'' means in this context, we have to make an assumption. 
\begin{assumption}[well-posedness of reference problem]\label{a:stabh}
Given $\k$, we assume that $V_h$ is chosen such that, for any $f\in V'$, the reference problem \eqref{e:modelh} admits a unique solution $u_h\in V_h$ that satisfies the polynomial-in-$\k$ stability bound
\begin{equation}\label{e:uhstab}
\|u_h\|_V\leq \Cpstab\k^n\|f\|_{\Omega},
\end{equation}
where the constant $\Cpstab$ is independent of $\kappa$ but may be different from the one in \eqref{e:stabpoly} (which is related to the stability of the full problem \eqref{e:modelvar}).
\end{assumption}

One example, where the range of feasible fine scale parameter $h$ can be quantified is the case of a pure Robin problem in a convex domain discretized by $P_1$ finite elements. In this case, the resolution condition $h\k^{3/2}\lesssim 1$ implies Assumption~\ref{a:stabh} \cite{Wu01072014}. A very limited number of further settings and methods is discussed in the literature (e.g. \cite{MelenkEsterhazy,hiptmair}) but in the majority of scenarios, such a quantification is completely open.
Still, assuming that $h$ is sufficiently fine, we are able to show the stability of the practical method. 
\begin{theorem}[stability and error of the fully discrete method]\label{t:errorloch}
If the fine scale discretization space $V_h$ is sufficiently rich so that Assumption~\ref{a:stabh} holds and if the coarse mesh width $H$ is sufficiently small in the sense of Assumption~\ref{a:resolution} ($H\k\lesssim 1$) and if the oversampling parameter $\ell\gtrsim\log\k$ is sufficiently large in the sense of \eqref{e:oversampling} (with constant $\Cpstab$ from \eqref{e:uhstab}), then the fully discrete localized Petrov-Galerkin approximation $\uplh\in\Vplh$ satisfies the error estimate
\begin{equation}\label{e:errorloch}
\|u_h-\uplh\|_V\leq  C(H+\Cloc{\ell}\beta^{\ell}\k^\bstab)\|f\|_\Omega,
\end{equation}
where $u_h$ solves the reference problem \eqref{e:modelh} and $C$ is some generic constant that does not depend on $H$, $\ell$ and $\k$. 
\end{theorem}
\begin{proof}
The proof follows closely the analysis of Section~\ref{s:loc} and simply replaces the space $V$ by $V_h$ in the construction  of the method and its error analysis. Almost all arguments remain valid. The only technical issue is that the space $V_h$ is not closed under multiplication by cut-off functions used in the proofs of Theorem~\ref{t:exp}, Lemma~\ref{l:localization}, and Theorem~\ref{t:errorlocalization}. 
This requires minor modifications as they have already been applied successfully in previous papers \cite{MP14,HP12,HMP14}. 
To begin with, let all cut-off functions $\eta$ be replaced by their nodal interpolation $\Inodal \eta$ on the coarse mesh $\tri_H$. This may affect the constant in \eqref{e:cutoffH} but not the overall results. This choice shows that $\eta\psi$ is piecewise polynomial with respect to the fine mesh $\tri_h$ and can be approximated by nodal interpolation $\Ih(\eta\psi)$ on the same mesh in a stable way.  
One example where such a modification is required is \eqref{e:decay1} in the proof of Theorem~\ref{t:exp}. The modification causes an additional term that measures the distance of $\eta\psi$ to the finite element space $V_h$,
\begin{align}
   \|\nabla \psi\|_{\Omega\setminus\Omega_{T,\ell-3}}^2 
   &= \Re\, (\nabla\psi,\nabla\left(\eta\psi\right))_{\Omega} - \Re\, (\nabla\psi,\psi\nabla \eta)_{\Omega}\nonumber\\
   &\leq |\Re\, (\nabla\psi,\nabla\left(\Ih(\eta\psi)-\Iloc(\IH(\Ih(\eta\psi)))\right))_{\Omega}|\nonumber\\
   &\qquad +|\Re\,(\nabla\psi,\nabla\Iloc(\IH(\Ih(\eta\psi))))_{\Omega}| + \left|\Re\, (\nabla\psi,\psi\nabla \eta)_{\Omega}\right|\nonumber\\
  &\qquad + |\Re\,(\nabla\psi,\nabla\left(\eta\psi-\Ih(\eta\psi)\right))_{\Omega}|\nonumber\\
 & =: \tilde{M}_1+\tilde{M}_2+\tilde{M}_3+\tilde{M}_4.\label{e:decay1h}
\end{align}
The treatment of $\tilde{M}_1, \tilde{M}_2, \tilde{M}_3$ is very similar to the treatment of $M_1,M_2,M_3$ in the proof of Theorem~\ref{t:exp} and requires only the stability of $I_h$ on the space of piecewise polynomials. 
Since $I_h(\eta\psi)=\psi$ outside the support of $\nabla \eta$, $\tilde{M}_4$ can easily be bounded by
\begin{equation*}
 \tilde{M}_4\leq (1+\Cint\sqrt{\Col})\|H\nabla\eta\|_{L^\infty(\Omega)})\|\nabla\phi\|_{\Omega_{T,\ell-2}\setminus\Omega_{T,\ell-5}}^2 
\end{equation*}
and further arguments remain valid (with a possible change of the constants involved). 
The proofs of Lemma~\ref{l:localization} and Theorem~\ref{t:errorlocalization} can be modified in a similar way.
\end{proof}
\begin{remark}[true errors] Note that Theorem~\ref{e:errorloch} compares the discrete localized Petrov-Galerkin approximation  with the reference solution $u_h$ \eqref{e:modelh} only. An estimate of the full error reads
\begin{equation*}
\|u-\uplh\|_V\leq  C(H+\Cloc{\ell}\beta^{\ell}\k^\bstab)\|f\|_\Omega + \|u-u_h\|_V.
\end{equation*}
Further estimation of the reference error $\|u-u_h\|_V$ relies on additional regularity of the solution $u$ in the usual way. As with stability, quantitative results are rare \cite{melenk_phd,MelenkEsterhazy,hiptmair}
In the simplest case of a pure Robin problem in a convex domain, an estimate of the form
\begin{equation*}
\|u-\uplh\|_V\lesssim (H+\beta^{\ell}\k + h\k)\|f\|_\Omega
\end{equation*}
holds true in the regime where $h\approx\k^{-2}$ (cf. \cite{melenk_phd}). In this regime, the choices $H\approx\k^{-1}$ and $\ell\approx 2\log(\k)$ would balance the three term in the error bound and the overall error of the method would be proportional to $\k^{-1}$.  
\end{remark}
Using the properties of the very particular example of the previous remark, the complexity of our approach may be estimated as follows. The number of degrees of freedom in a single cell problem scales like $(\ell H/h)^d$. Although the cell problems are non-hermitian, they are coercive and linear complexity with respect to the number of degrees of freedom (independent of $\kappa$) can be expected for a suitably chosen multilevel preconditioned iterative solver. Let us assume that $m$ local problems need to be solved where $m$ depends on the geometric setting of the problem and the structuredness of the meshes. E.g.,  $m\approx \ell\approx \log(\kappa)$ for a uniform mesh on the unit square. Then the cost of pre-computing the bases up to a given accuracy and assembling the coarse problem is roughly $\mathcal{O}(\log(\kappa)^{d+1}(\kappa)^{d})$. This cost does not exceed the cost for solving the resulting $\mathcal{O}(\kappa^{d})$-dimensional coarse Helmholtz problem because this system faces the indefiniteness of the Helmholtz problem in the usual \cite{EG12,GGS15} which makes it difficult to design algebraic solvers of optimal complexity and, in particular, solvers that are robust with respect to large wave numbers. In general, the complexity of the method may as well be dominated by  the fine scale computation (depending on the stability problems of the original problem). In this case, the inherent independence of the corrections and their independence of the right-hand side and boundary data are interesting features of the method that can increase the efficiency of computations in the context of large-scale and inverse problems.

\section{Numerical Experiments}\label{s:numexp}
In this section we will present two numerical examples. We apply our method to model Helmholtz problems in one and two dimensions and compare the results with standard $P_1$ finite elements. We will demonstrate the validity of our estimates based on varying oversampling parameter $\ell$, coarse mesh size $H$ and by varying the wave number $\k$. A
comprehensive numerical study of the algorithmic ideas proposed in this paper
is topic of current and future research.

\subsection{Illustration of the theoretical results in $1d$}\label{ss:numexp1}
Let $\Omega:=(0,1)$, $\Gamma_R=\partial\Omega$ (solely Robin boundary condition), and let the right-hand side $f$ defined by
\begin{align}\label{e:numexp1}
f(x) &:= \begin{cases}
     2\sqrt{2},& x\in[\tfrac{3}{16},\tfrac{5}{16}]\cup[\tfrac{11}{16},\tfrac{13}{16}],\\
     0,& \text{elsewhere,}
       \end{cases}
\end{align}
represent two radiating sources. The right-hand side was normalized so that $\|f\|_{L^2(\Omega)}=1$. 

Note that this one-dimensional example does not serve as a proper benchmark for the method because there even exist local pollution-free generalized finite element methods. Actually, in $1d$ (and only in $1d$) the choice of the nodal interpolation for $I_H$ would have lead to such a method. Still, this model problem nicely reflects our theoretical results for a wide range of wave numbers. Since non of our arguments depends on the space dimension (though some constants do), the $1d$ performance truly illustrates the performance that can be observed also in higher dimensions. 

We consider the following values for the wave number, $\k=2^3,2^4,\ldots,2^7$. The numerical experiment aims to study the dependence between these wave numbers and the accuracy of the numerical method. Consider the equidistant coarse meshes with mesh widths $H=2^{-1},\ldots,2^{-10}$. The reference mesh $\tri_h$ is derived by uniform mesh refinement of the coarse meshes and has maximal mesh width $h=2^{-14}$. The corresponding $P_1$ conforming finite element approximation on the reference mesh $\tri_h$ is denoted by $V_h$. We consider the reference solution $u_h\in V_h$ of \eqref{e:modelh} with data given in \eqref{e:numexp1} and compare it with coarse scale approximations $\uplh\in \Vplh$ (cf. Definition~\ref{e:VclhGalerkin}) depending on the coarse mesh size $H$ and the oversampling parameter $\ell$.

The results are visualized in Figures~\ref{fig:numexp1H} and~\ref{fig:numexp1k}. Figure~\ref{fig:numexp1H_msfem} shows the relative energy errors $\frac{\|u_h-\uplh\|_V}{\|u_h\|_V}$ depending on the coarse mesh size $H$ for several choices of the wave number $\k=2^3,2^4,\ldots,2^7$. The oversampling parameter $\ell$ is tied to $H$ via the relation $\ell=\ell(H)=|\log_2 H|$. This choice seems to be sufficient to preserve optimal convergence as soon as $H\kappa\lesssim 1$ holds. The experimental rate of convergence $N_{\operatorname{dof}}^{3/(2d)}$ is better than predicted by Theorem~\ref{t:errorloch}. This effect is due to some unexploited $L^2$-orthogonality properties of the quasi-interpolation operator $\IH$; see \cite[Section 2]{MR1736895} and \cite[Remark 3.2]{MP14} for details. In the regime $H\kappa\lesssim 1$, the errors coincide to with those of the best approximation (with respect to the $V$-norm) of $u_h$ in the space $\Vplh$ depicted in Figure~\ref{fig:numexp1H_msfema}. 

We also show errors of the Petrov-Galerkin method based on the pairing $(V_H,\Vdlh)$ (the localized and fully discretized version of \eqref{e:Galerkinglobalstab}) in Figure~\ref{fig:numexp1H_msfemP1}. The stabilization via the precomputed test functions cures pollution and the errors are comparable to those of the best-approximation among the $P_1$-finite element functions\footnote{While the present paper was still under review, this observation has been justified theoretically in the follow-up paper \cite{Gallistl.Peterseim:2015}; see also \cite{Peterseim2015}.}, whereas the pollution effect is clearly visible for the standard conforming $P_1$-FEM (Galerkin) on the coarse meshes; see  Figure~\ref{fig:numexp1H_msfemP1a}. 
\begin{figure}[tb]
\begin{center}
\subfigure[\label{fig:numexp1H_msfem}Results for multiscale method \eqref{e:VclhGalerkin}.
]{\includegraphics[width=0.49\textwidth]{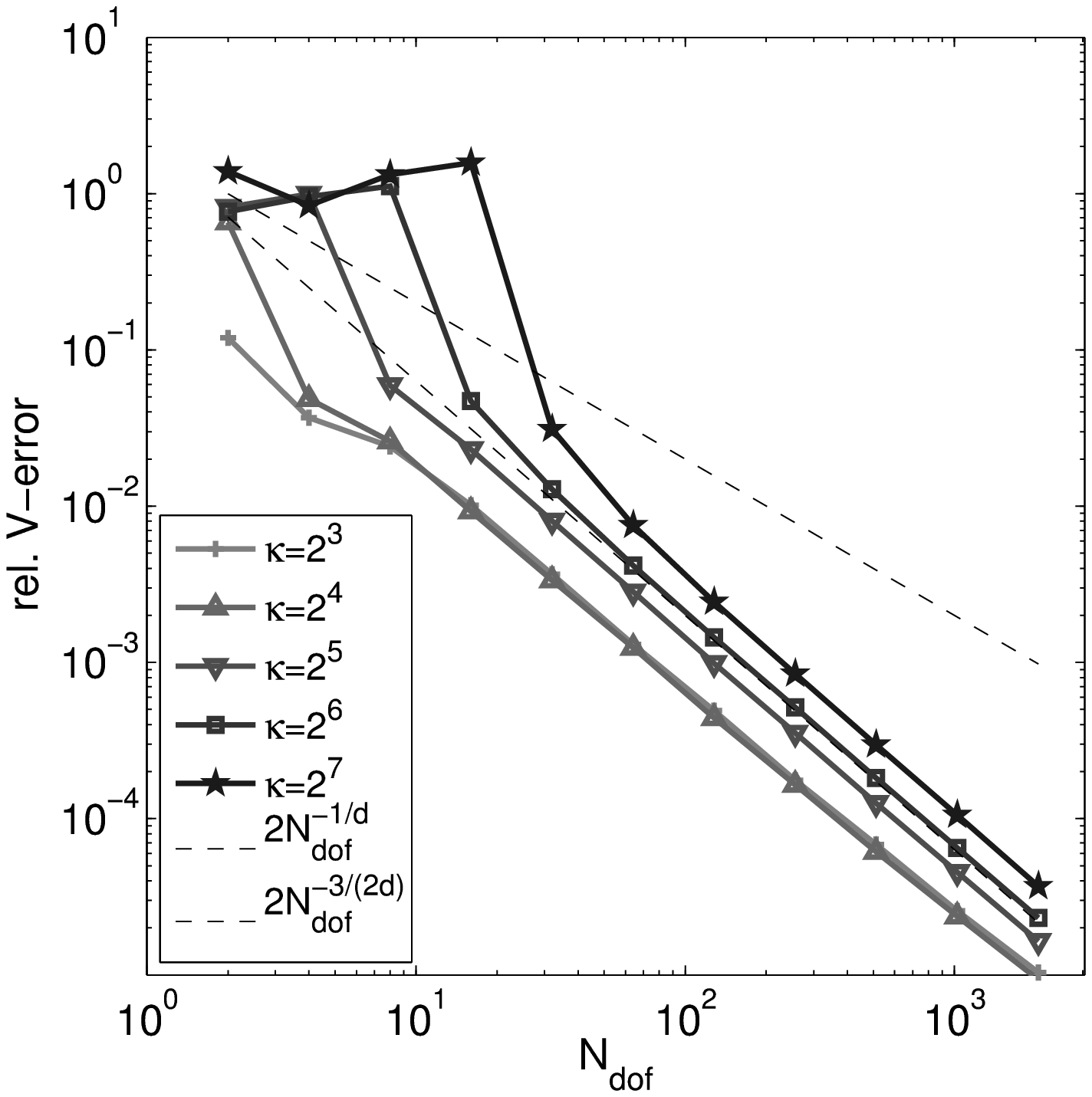}}
\subfigure[\label{fig:numexp1H_msfema}Results for $V$-best-approximation in $\Vplh$.
]{\includegraphics[width=0.49\textwidth]{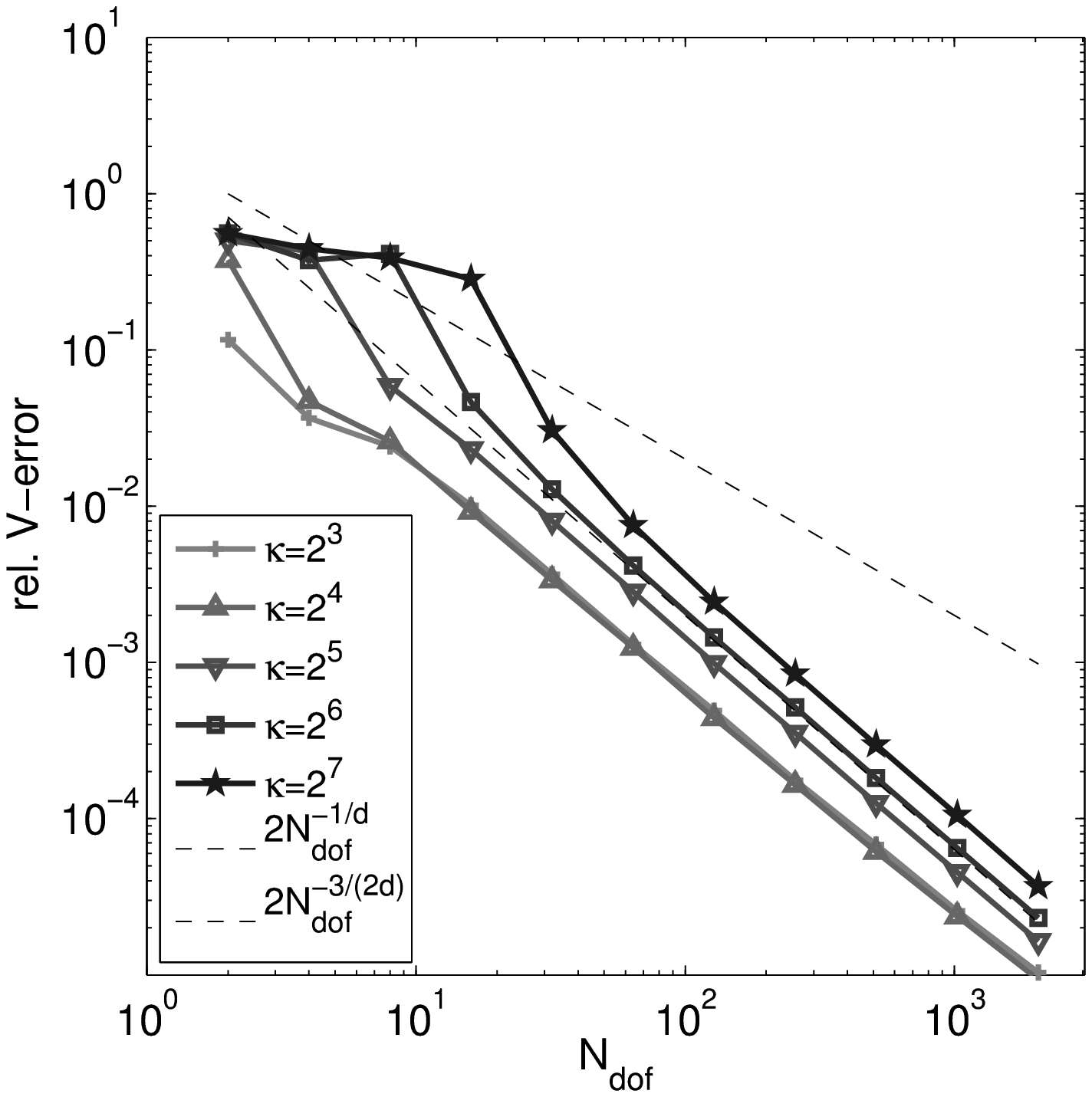}}\\
\subfigure[\label{fig:numexp1H_msfemP1}Results for multiscale Petrov-Galerkin method with trial space $V_H$ and test space $\Vdlh$.
]{\includegraphics[width=0.49\textwidth]{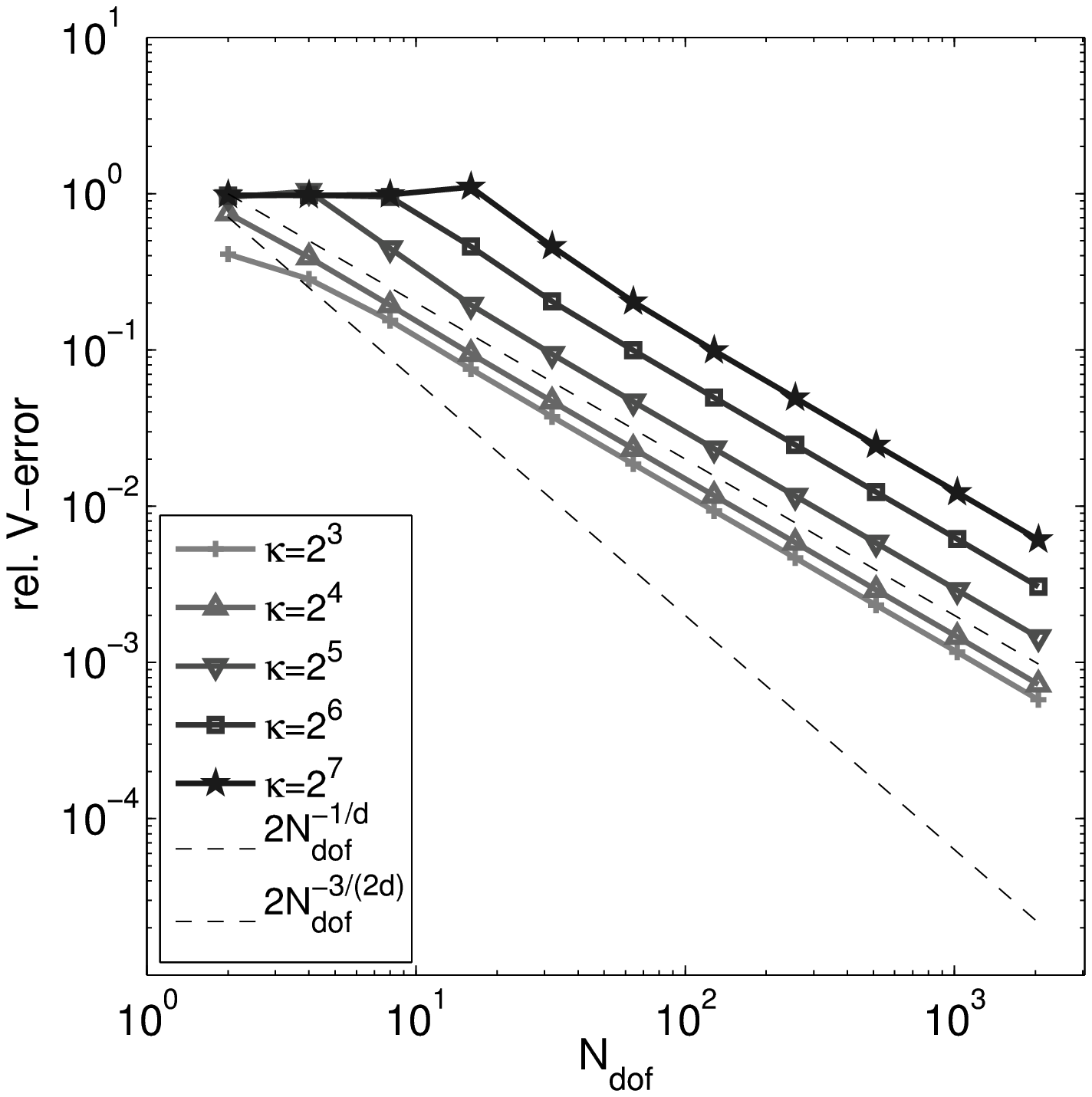}}
\subfigure[\label{fig:numexp1H_msfemP1a}Results for standard Galerkin in the space $V_H$.
]{\includegraphics[width=0.49\textwidth]{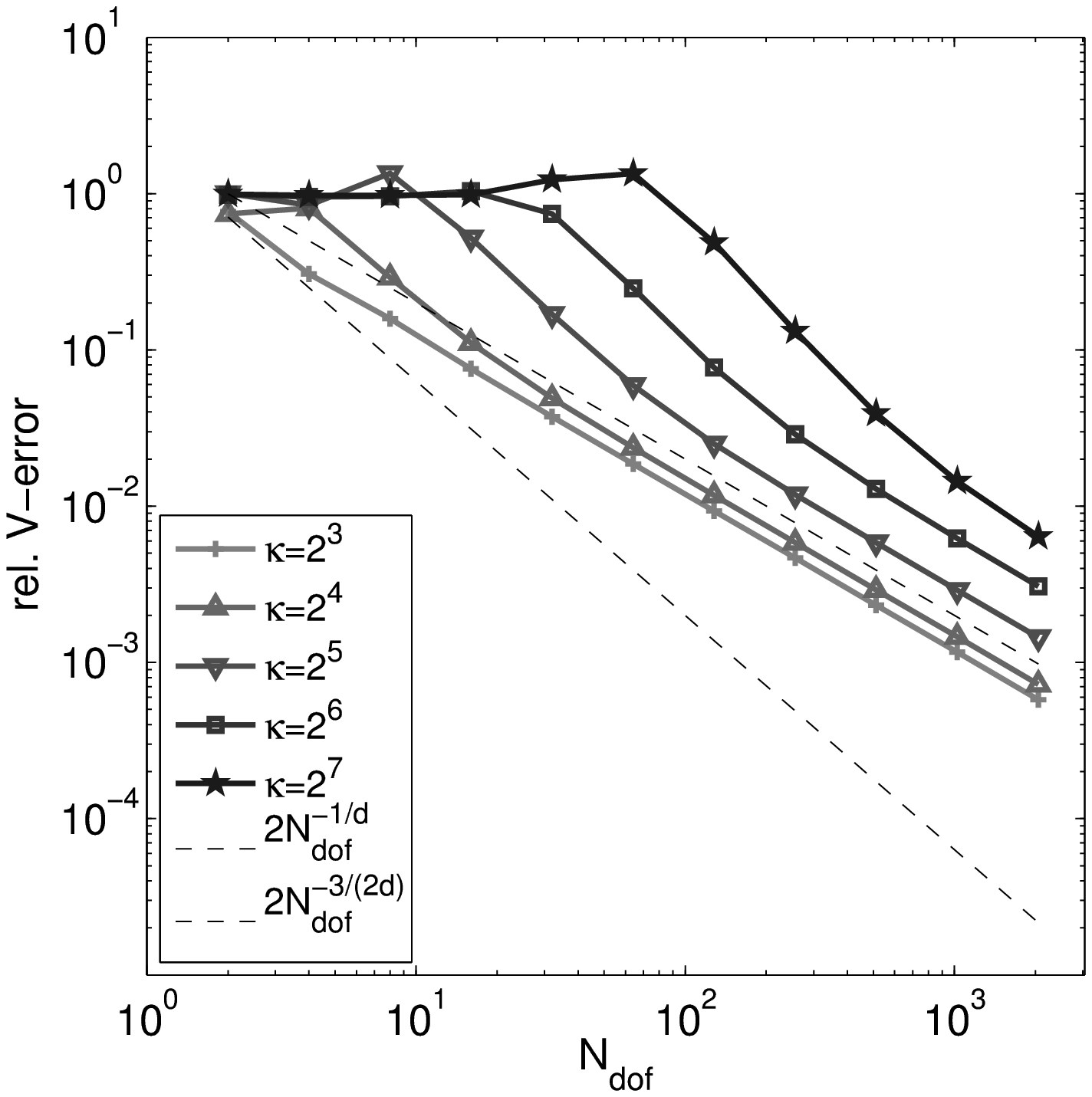}}
\end{center}
\caption{Numerical experiment of Section~\ref{ss:numexp1}: Results for the multiscale method \eqref{e:VclhGalerkin}, a modification based on the trial-test-pairing $(V_H,\Vdlh)$ and standard $P_1$-FEM with several choices of the wave number $\k$ depending on the uniform coarse mesh size $H=N_{\operatorname{dof}}^{-1}$. The reference mesh size $h=2^{-14}$ remains fixed. The oversampling parameter is tied to the coarse mesh size via the relation $\ell=|\log_2 H|$ in (a)-(c). \label{fig:numexp1H}}
\end{figure}

Figure~\ref{fig:numexp1k} aims to illustrate the role of the oversampling parameter. It depicts the relative energy errors $\frac{\|u_h-\uplh\|_V}{\|u_h\|_V}$ of the method \eqref{e:VclhGalerkin} and the best-approximation in $\Vplh$ depending on the coarse mesh size $H$ for fixed wave number $\k=2^7$ and several choices of the oversampling parameter $\ell=1,2,3,\ldots,8$. (We also show errors of the standard conforming $P_1$-FEM on the coarse meshes for comparison.) The exponential decay of the error with respect to $\ell$ is observed once the mesh size reaches the regime of resolution $H\k\lesssim 1$. Moreover, Figure~ \ref{fig:numexp1k_msfema} shows that, for fixed $\ell$, the approximation property of $\Vplh$ does not improve with decreasing $H$ and the oversampling parameter needs to be increased with decreasing $H$ to get any rate. By contrast, the Petrov-Galerkin method based on the trial-test-pairing $(V_H,\Vplh)$ (which in fact computes $\Pi_H u_h$ for $\ell\rightarrow\infty$) allows to reduce the oversampling parameter with decreasing $H\k$ until, for $H\k^2\approx1$, the correction can be removed because $P_1$-FEM becomes quasi-optimal; see Figure~ \ref{fig:numexp1k_msfemb} which depicts relative $L^2$-errors of the method.
\begin{figure}[tb]
\begin{center}
\subfigure[\label{fig:numexp1k_msfem}Results for multiscale method \eqref{e:VclhGalerkin}.
]{\includegraphics[width=0.49\textwidth]{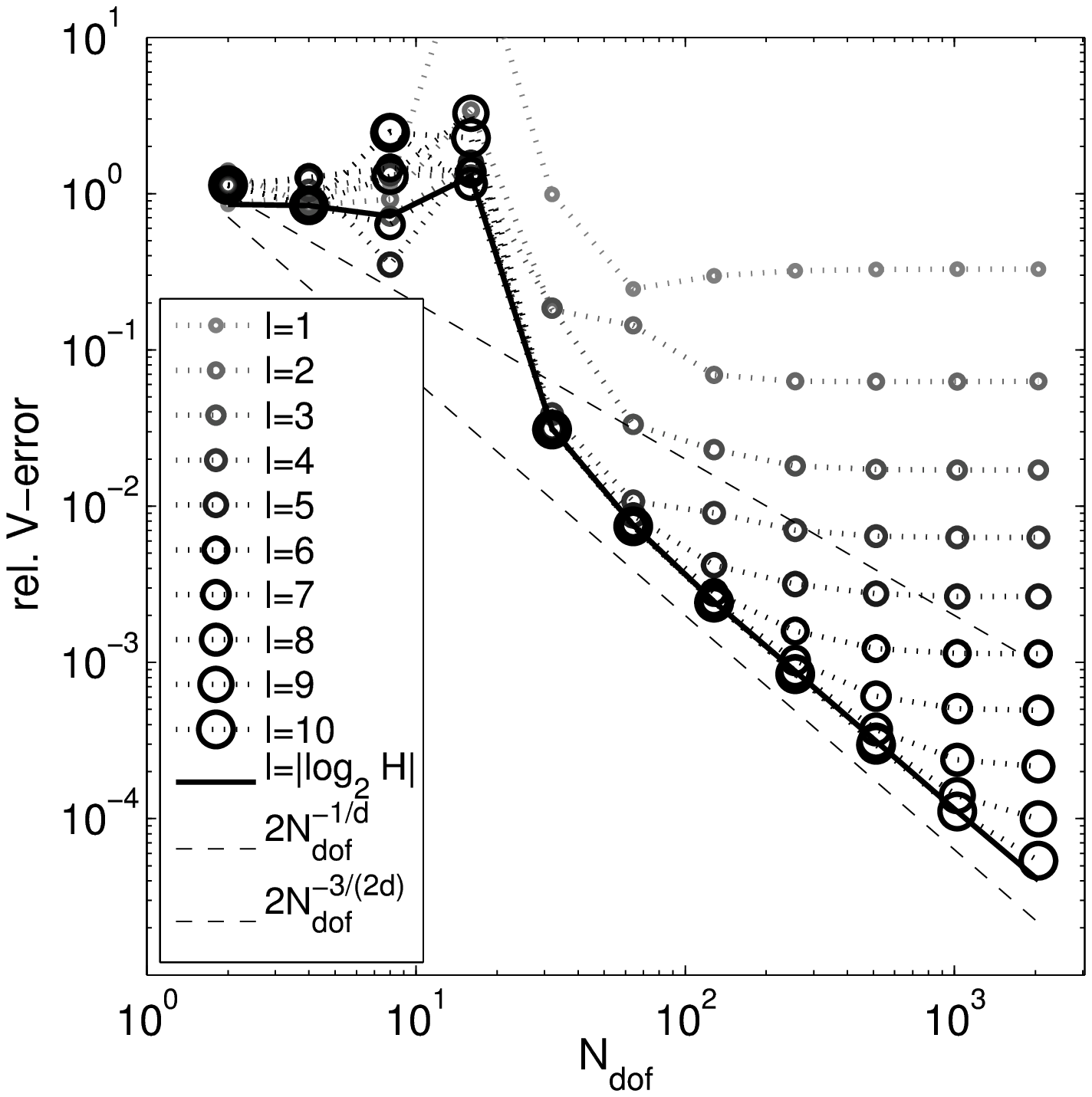}}
\subfigure[\label{fig:numexp1k_msfema}Results for $V$-best-approximation in $\Vplh$.
]{\includegraphics[width=0.49\textwidth]{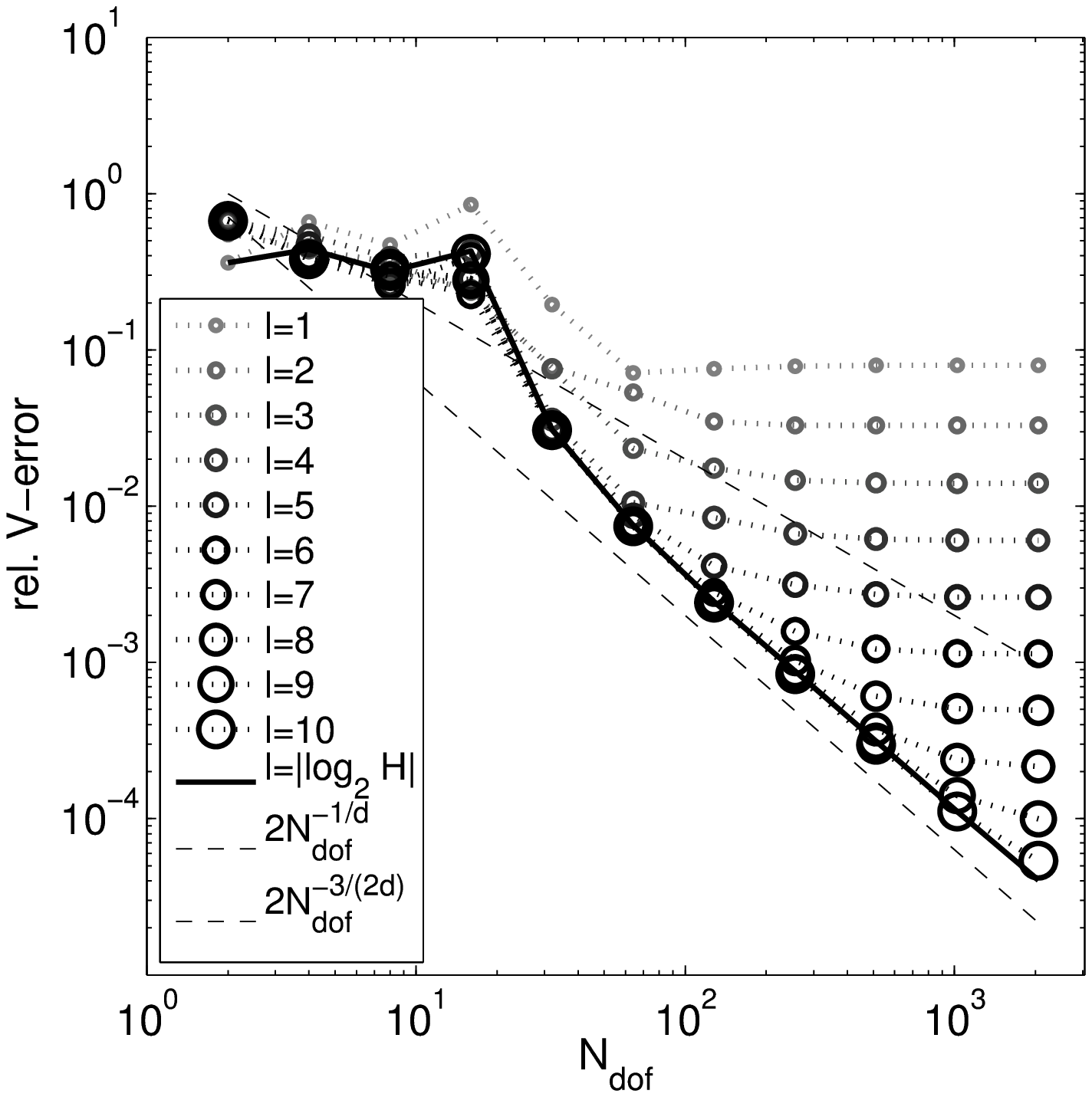}}
\subfigure[\label{fig:numexp1k_msfemb}Results for multiscale Petrov-Galerkin method with trial space $V_H$ and test space $\Vplh$.
]{\includegraphics[width=0.49\textwidth]{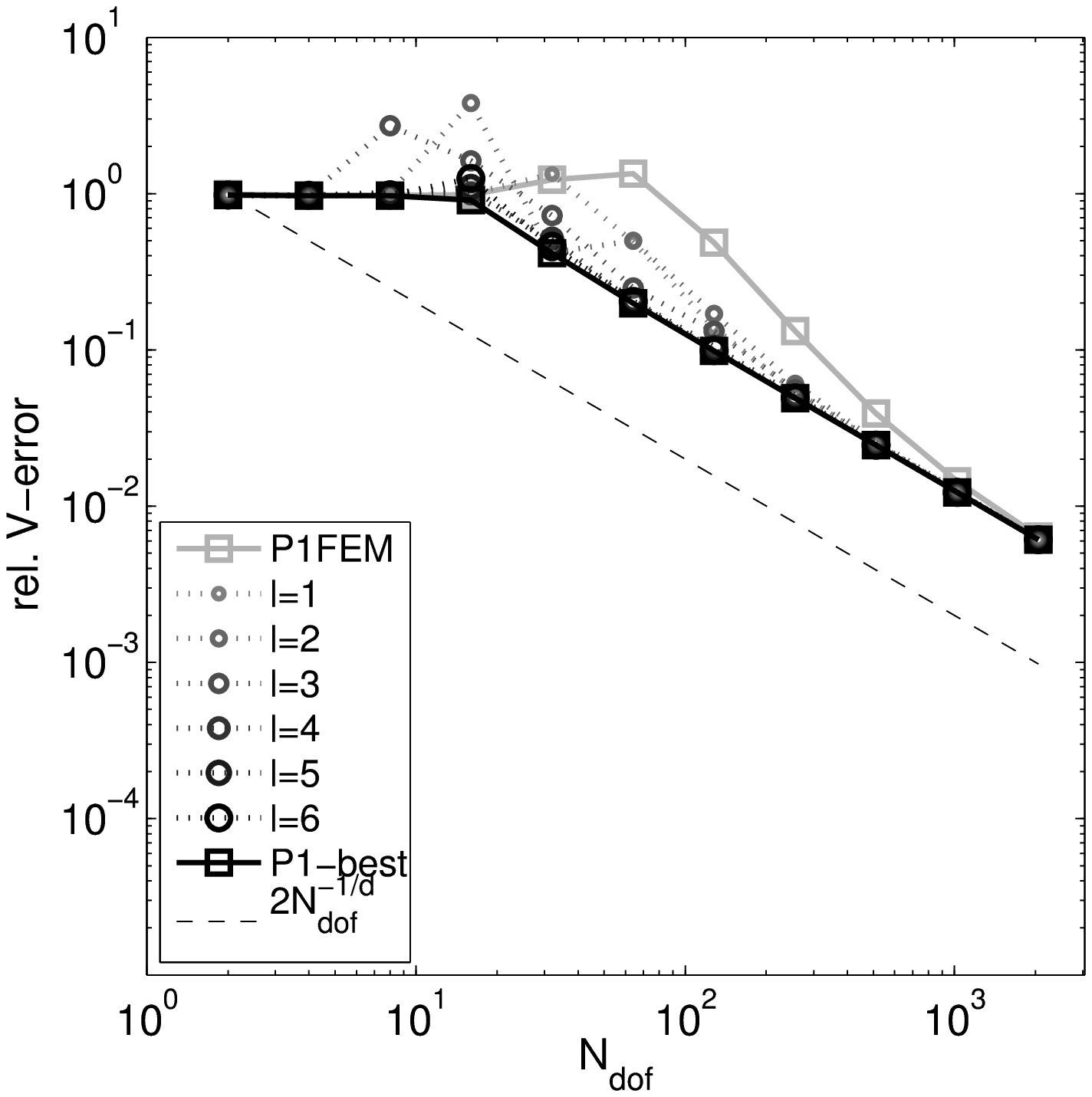}}
\end{center}
\caption{Numerical experiment of Section~\ref{ss:numexp1}: Results for multiscale method \eqref{e:VclhGalerkin} with wave number $\k=2^8$ depending on the uniform coarse mesh size $H=N_{\operatorname{dof}}^{-1}$. The reference mesh size $h=2^{-14}$ remains fixed. The oversampling parameter $\ell$ varies between $1$ and $10$. \label{fig:numexp1k}}
\end{figure}

Finally, we want to show that a different choice of interpolation operator in the definition \eqref{e:RH} of the remainder space can lead to very different practical performance (within the range of the theoretical predictions though). Figure~\ref{fig:numexp1kQ} shows the results for the above experiment when the operator $\mathcal{Q}_H$ from \eqref{e:QH} is used instead of $\IH$. It turns out that, for this example, the decay of the correctors is much faster so that the same accuracy is achieved with basis functions that are far more localized. A similar observation has been made previously in the context of high-contrast diffusion problems \cite{Brown.Peterseim:2014,Peterseim.Scheichl:2014}. This improved performance can not be quantified by the existing theory and requires further investigation.
\begin{figure}[tb]
\begin{center}
\subfigure[\label{fig:numexp1kQ_msfem}Results for multiscale method \eqref{e:VclhGalerkin} based on quasi-interpolation $\mathcal{Q}_H$.
]{\includegraphics[width=0.49\textwidth]{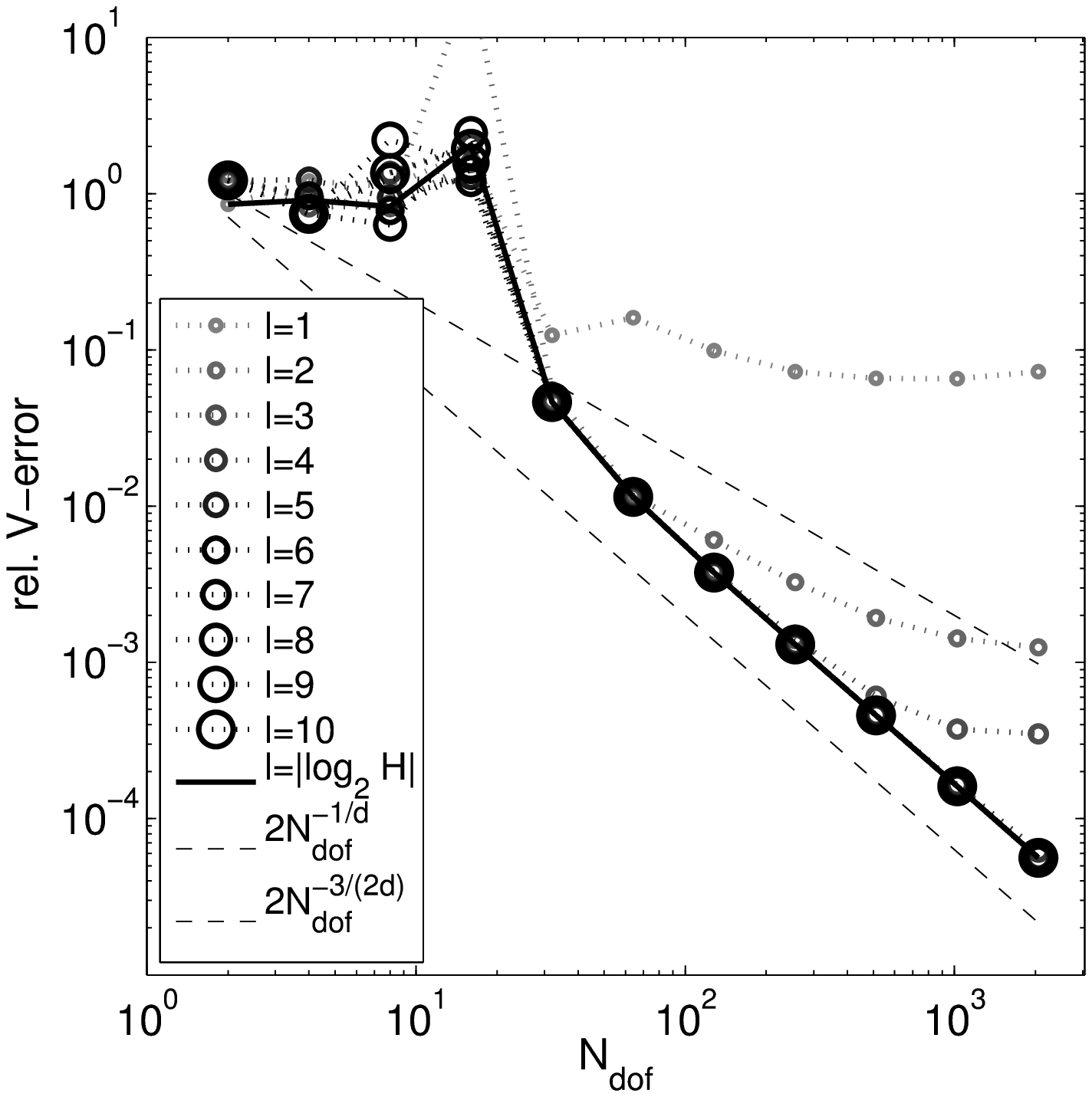}}
\subfigure[\label{fig:numexp1kQ_msfema}Results for $V$-best-approximation in $\Vplh$ based on quasi-interpolation $\mathcal{Q}_H$.
]{\includegraphics[width=0.49\textwidth]{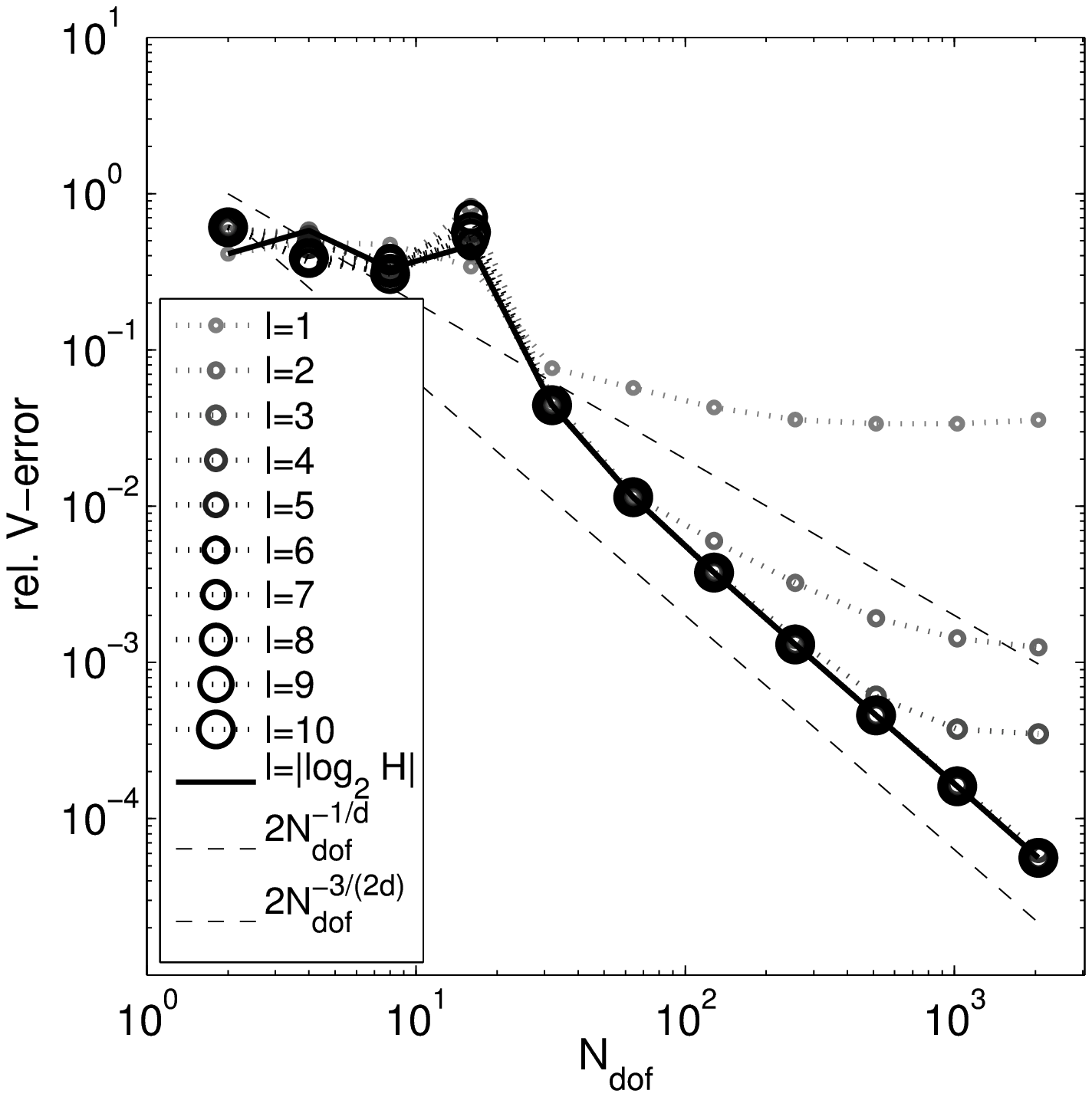}}
\subfigure[\label{fig:numexp1kQ_msfemb}Results for multiscale Petrov-Galerkin method with trial space $V_H$ and test space $\Vplh$ based on quasi-interpolation $\mathcal{Q}_H$.
]{\includegraphics[width=0.49\textwidth]{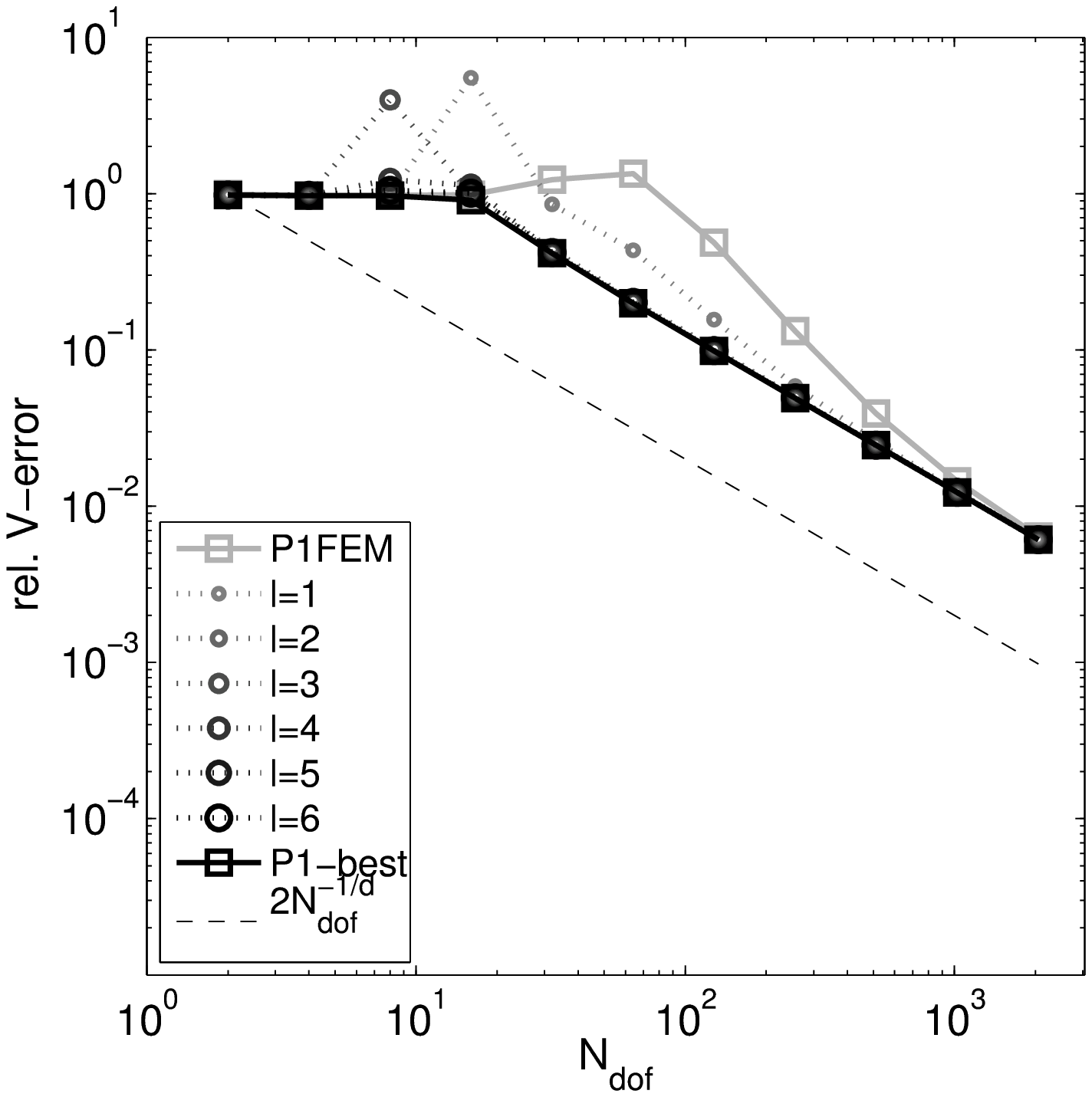}}
\end{center}
\caption{Numerical experiment of Section~\ref{ss:numexp1}: Results for multiscale method \eqref{e:VclhGalerkin} with interpolation operator $\mathcal{Q}_H$ for wave number $\k=2^8$ depending on the uniform coarse mesh size $H=N_{\operatorname{dof}}^{-1}$. The reference mesh size $h=2^{-14}$ remains fixed. The oversampling parameter $\ell$ varies between $1$ and $8$. \label{fig:numexp1kQ}}
\end{figure}

\subsection{Scattering from a triangle}\label{ss:numexp2}
The second experiment considers the scattering from sound-soft scatterer occupying the triangle $\Omega_D$. The Sommerfeld radiation condition of the scattered wave is approximated by the Robin boundary condition on the boundary $\Gamma_R:=\partial\Omega_R$ of the artificial domain $\Omega_R:=]0,1[^2$ so that $\Omega:=\Omega_R\setminus\Omega_D$ is the computational domain; see Figure~\ref{fig:numexp2domain1}. The incident wave $\displaystyle u_{inc}(x):=\exp\left(i\k\;x\cdot \left(\begin{smallmatrix}\cos(1/2)\\\sin(1/2)\end{smallmatrix}\right)\right)$ is prescribed via an inhomogeneous Dirichlet boundary condition on $\Gamma_D:=\partial\Omega_D$ and the scattered wave satisfies the model problem \eqref{e:modela} with the boundary conditions
\begin{align*}
  u &= -u_{inc}\quad\text{on }\Gamma_D,\\
  \nabla u\cdot \nu - i\k u &= 0\quad\text{on }\Gamma_R.
\end{align*}
The error analysis of the previous sections extends to this setting in a straight-forward way. By introducing some function $u_0\in W^{1,2}(\Omega)$ that satisfies the above boundary conditions, the problem can be reformulated with homogenous boundary conditions and the additional term $-a(u_0,v)$ on the right side of \eqref{e:modelvar}. This corresponds to having the modified right hand side $f+\Delta u_0+\k^2 u_0$ in the strong form \eqref{e:modela} of the problem. If $u_0$ can be chosen such that $\Delta u_0\in L^2(\Omega)$, then all error bounds of this paper remain valid. For weaker right hand sides the rates with respect to $H$ are reduced accordingly. Note, however, that the $L^2$-norm of the modified right-hand side may depend on $\k$ as it is the case in the present experiment where $u_0$ is related to the incident wave. The best-approximation properties of the method (cf. Remark~\ref{r:quasioptimal}) are not affected by this possible $\kappa$-dependence of the errors. 

The numerical experiment considers the following values for the wave number, $\k=2^2,2^3,2^4,2^5$, and aims to study the dependence between the wave numbers and the accuracy of the numerical method. We choose uniform coarse meshes with mesh widths $H=2^{-2},\ldots,2^{-5}$ as depicted in Figures~\ref{fig:numexp2domain2}--\ref{fig:numexp2domain3}. The experiment focuses on the role of localization and keeps the fine scale $h=2^{-9}$ fixed. This is roughly $\kappa^{-2}$ for the largest $\kappa$ considered in the experiment so that there is some confidence that the fine scale error is not dominant and that the plotted reference errors also reflect the true errors, at least in the targeted regime $H\approx\kappa^{-1}$. For numerical experiments with variable fine-scale parameter $h$ we refer to \cite{Gallistl.Peterseim:2015}. The reference mesh $\tri_h$ is derived by uniform mesh refinement of the coarse meshes and has mesh width $h=2^{-9}$. 

As in the previous experiment, we consider the reference solution $u_h\in V_h$ of \eqref{e:modelh} with the above data and compare it with coarse scale approximations $\uplh\in \Vplh$ (cf. Definition~\ref{e:VclhGalerkin}) depending on the coarse mesh size $H$ and the oversampling parameter $\ell$. Here, we are using again the canonical quasi-interpolation $\IH$. Figures~\ref{fig:numexp2H} and \ref{fig:numexp2k} show the results which conform to the theoretical predictions. If the oversampling parameter is chosen appropriately ($\ell=|\log_2 H|$) then pollution effects are eliminated for both the multiscale method \eqref{e:VclhGalerkin} and for the Petrov-Galerkin method based on the trial-test-pairing $(V_H,\Vplh)$ -- the localized and fully discretized version of the stabilized method \eqref{e:Galerkinglobalstab}. Moreover, the low regularity of the solution does not affect the convergence rates of the multiscale method \eqref{e:VclhGalerkin} when compared with the reference solution $u_h$, whereas reduced rates are observed for the Petrov-Galerkin method based on the trial-test-pairing $(V_H,\Vplh)$ as expected (due to the limited approximation properties of $P_1$ functions in the Sobolev spaces $W^{s,2}(\Omega)$ for $s<11/7$). The regularity of the solution, however, does affect the accuracy of the reference solution $u_h$ and, hence, limits the overall accuracy of our approximation. The possibility of automatic balancing the local fine scale errors of the corrector problems, the localization error, the global coarse error, and further errors due to quadrature and inexact algebraic solvers is a desirable feature of the method that needs to be addressed by future research.
\begin{figure}[tb]
\begin{center}
\subfigure[\label{fig:numexp2domain1}Computational domain $\Omega$ with scatterer $\Omega_D$.
]{\includegraphics[width=0.32\textwidth]{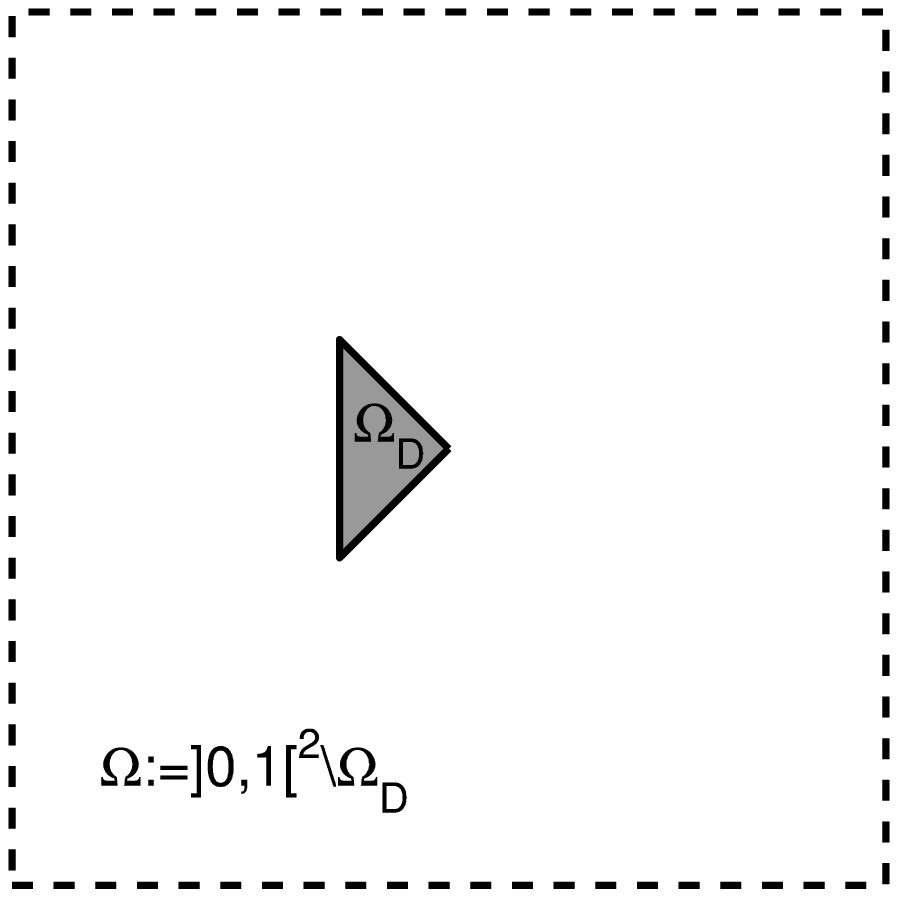}}
\subfigure[\label{fig:numexp2domain2}Initial coarse mesh.
]{\includegraphics[width=0.32\textwidth]{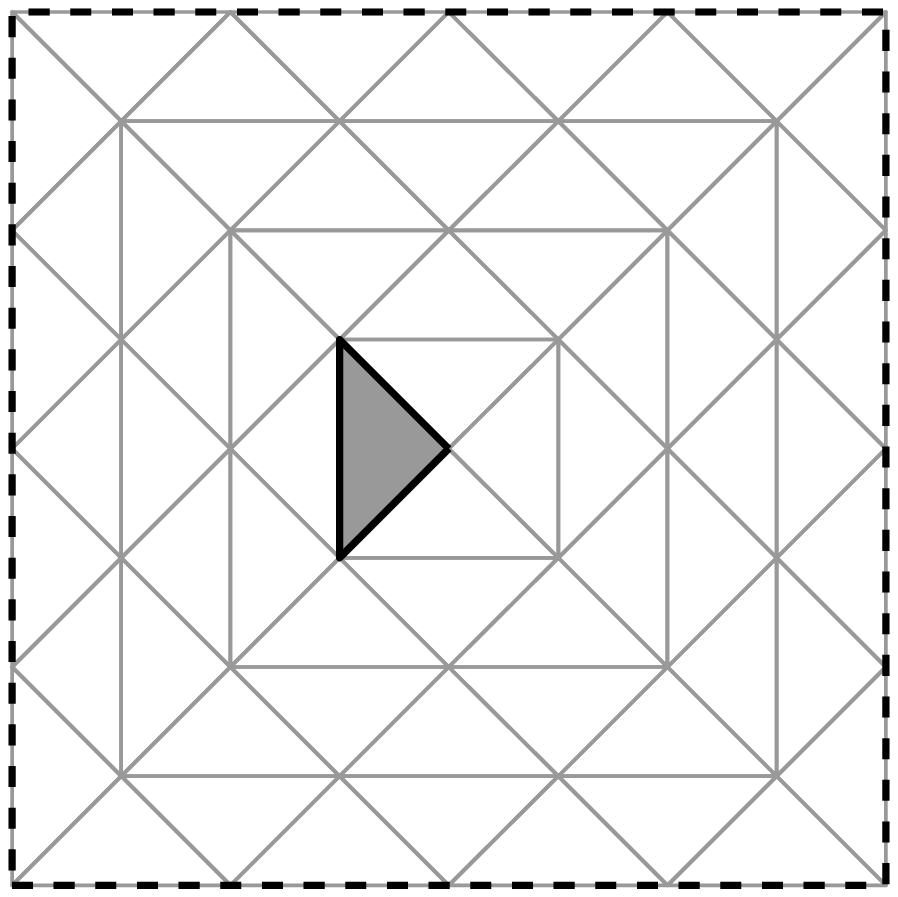}}
\subfigure[\label{fig:numexp2domain3}Uniformly refined coarse mesh.
]{\includegraphics[width=0.32\textwidth]{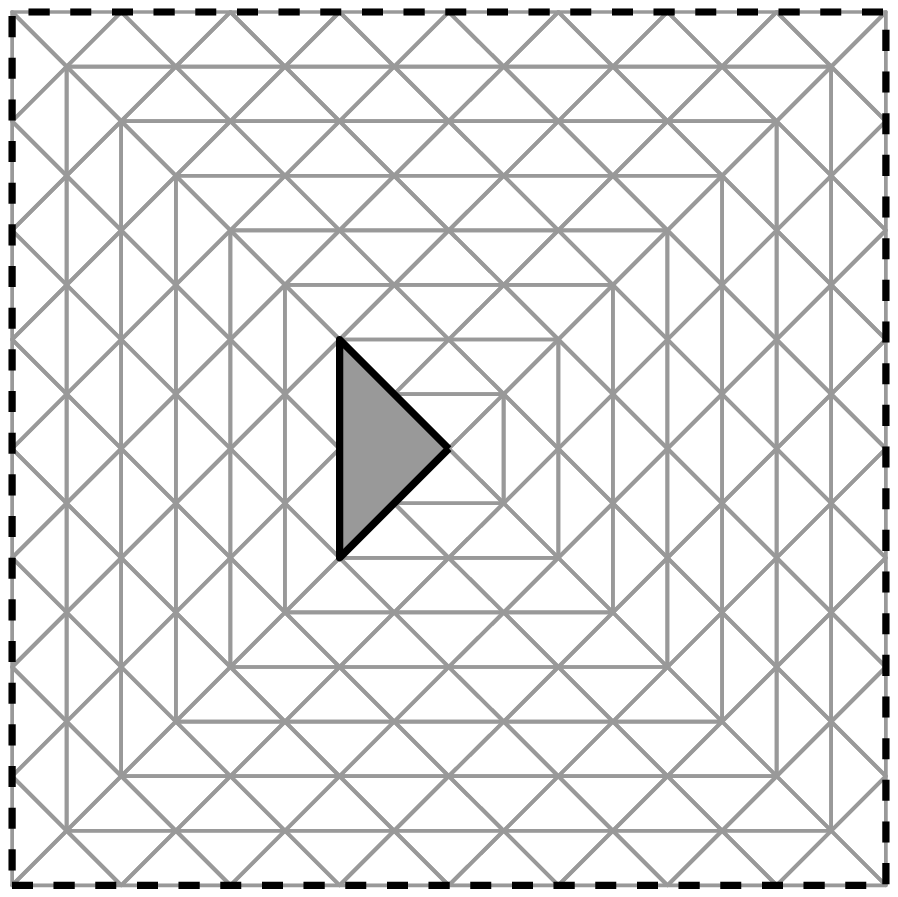}}
\end{center}
\caption{Computational domain of the model problem of Section~\ref{ss:numexp2} and corresponding coarse meshes. \label{fig:numexp2domain}}
\end{figure}

\begin{figure}[tb]
\begin{center}
\subfigure[\label{fig:numexp2H_msfem}Results for multiscale method \eqref{e:VclhGalerkin}.
]{\includegraphics[width=0.49\textwidth]{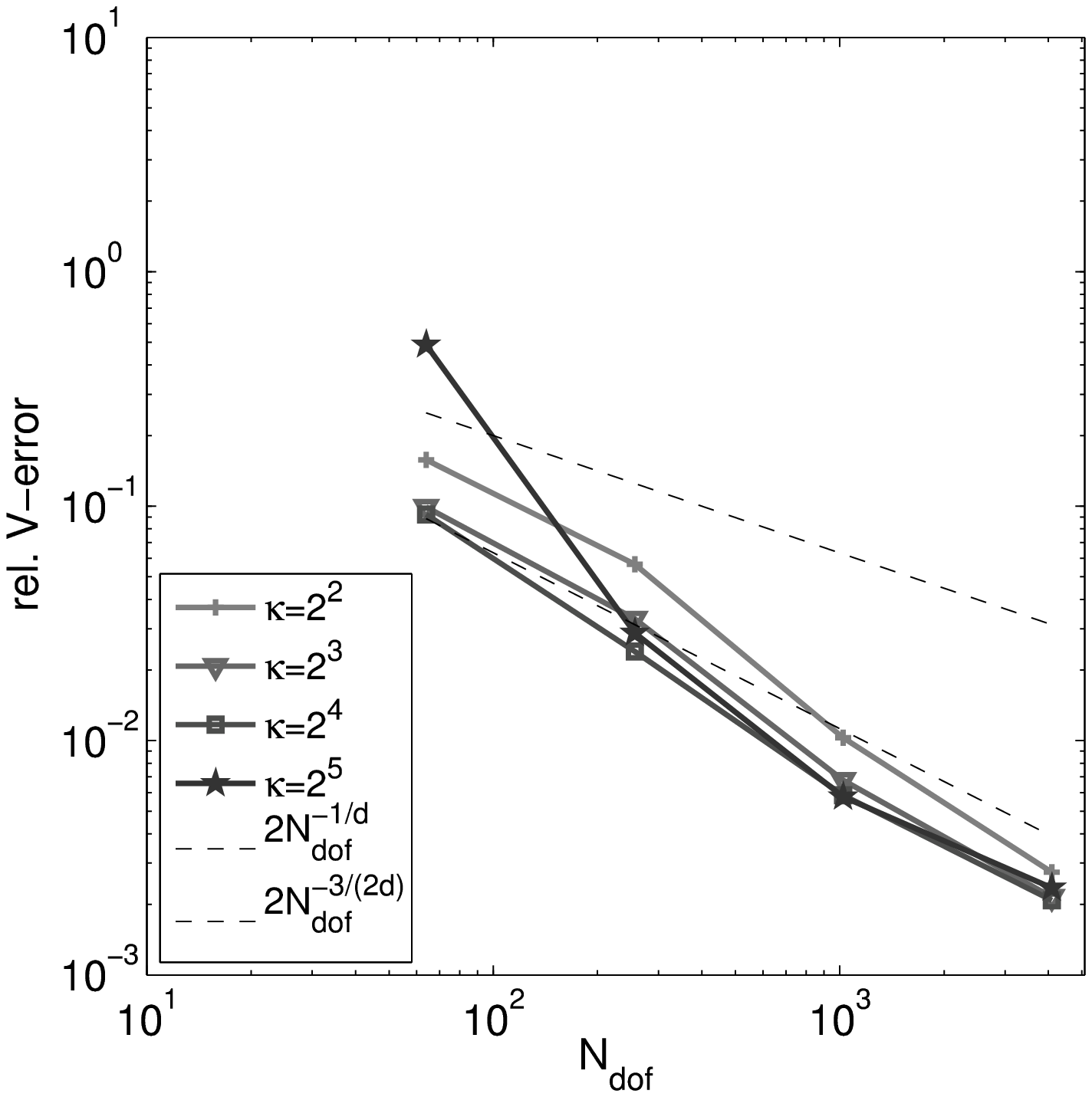}}
\subfigure[\label{fig:numexp2H_msfema}Results for $V$-best-approximation in $\Vplh$.
]{\includegraphics[width=0.49\textwidth]{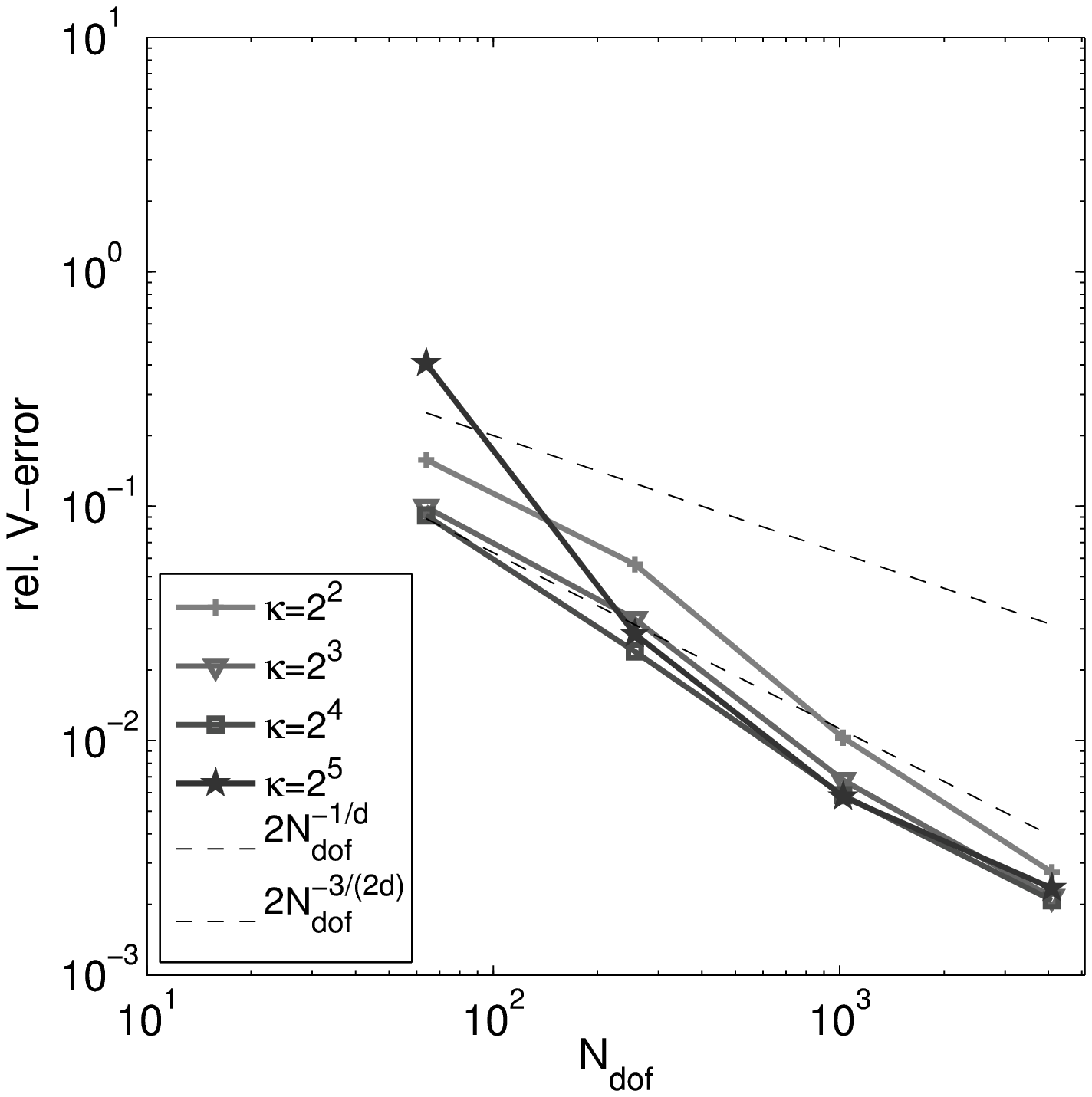}}\\
\subfigure[\label{fig:numexp2H_msfemP1}Results for multiscale Petrov-Galerkin method with trial space $V_H$ and test space $\Vdlh$.
]{\includegraphics[width=0.49\textwidth]{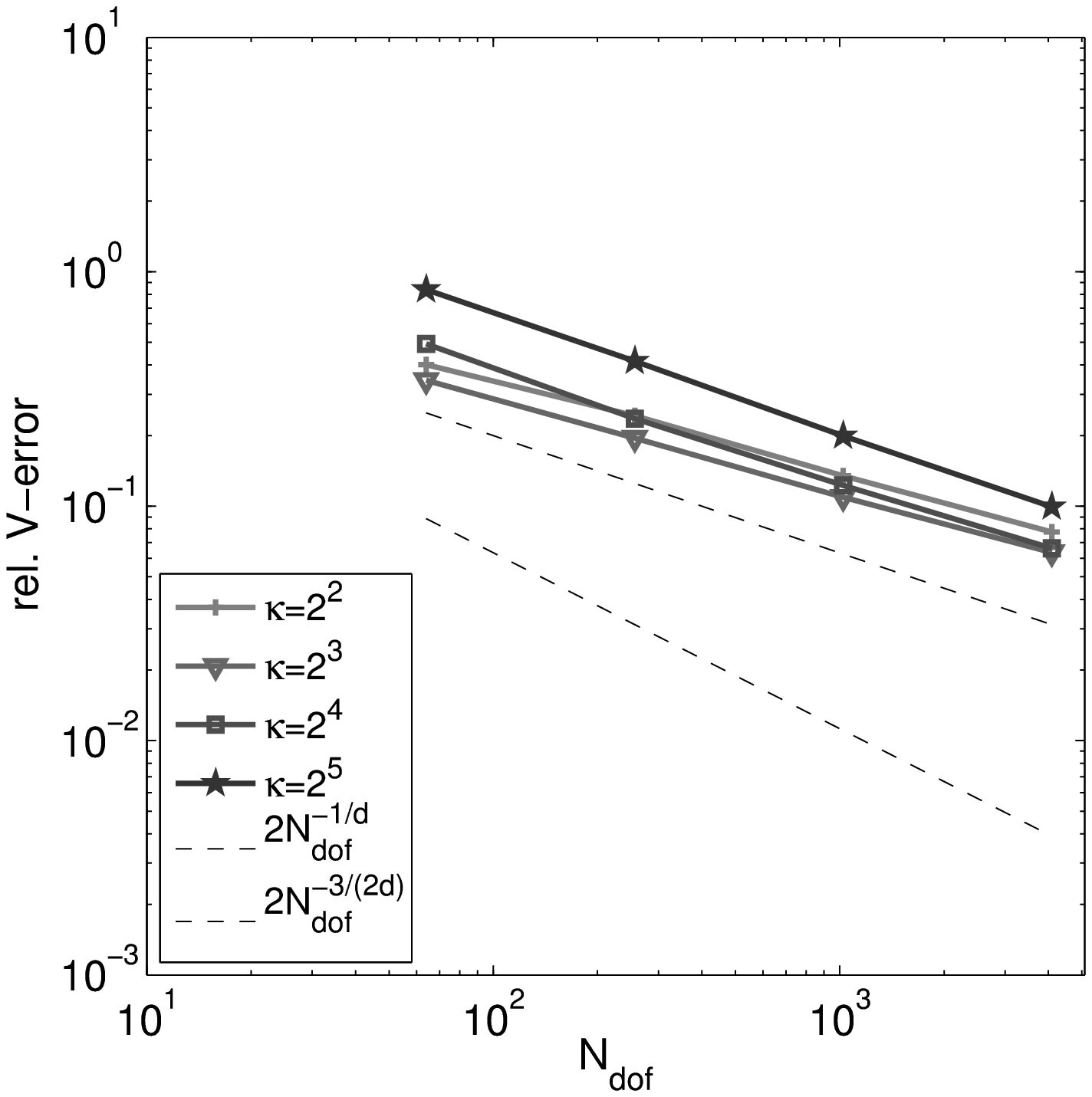}}
\subfigure[\label{fig:numexp2H_msfemP1a}Results for standard Galerkin in the space $V_H$.
]{\includegraphics[width=0.49\textwidth]{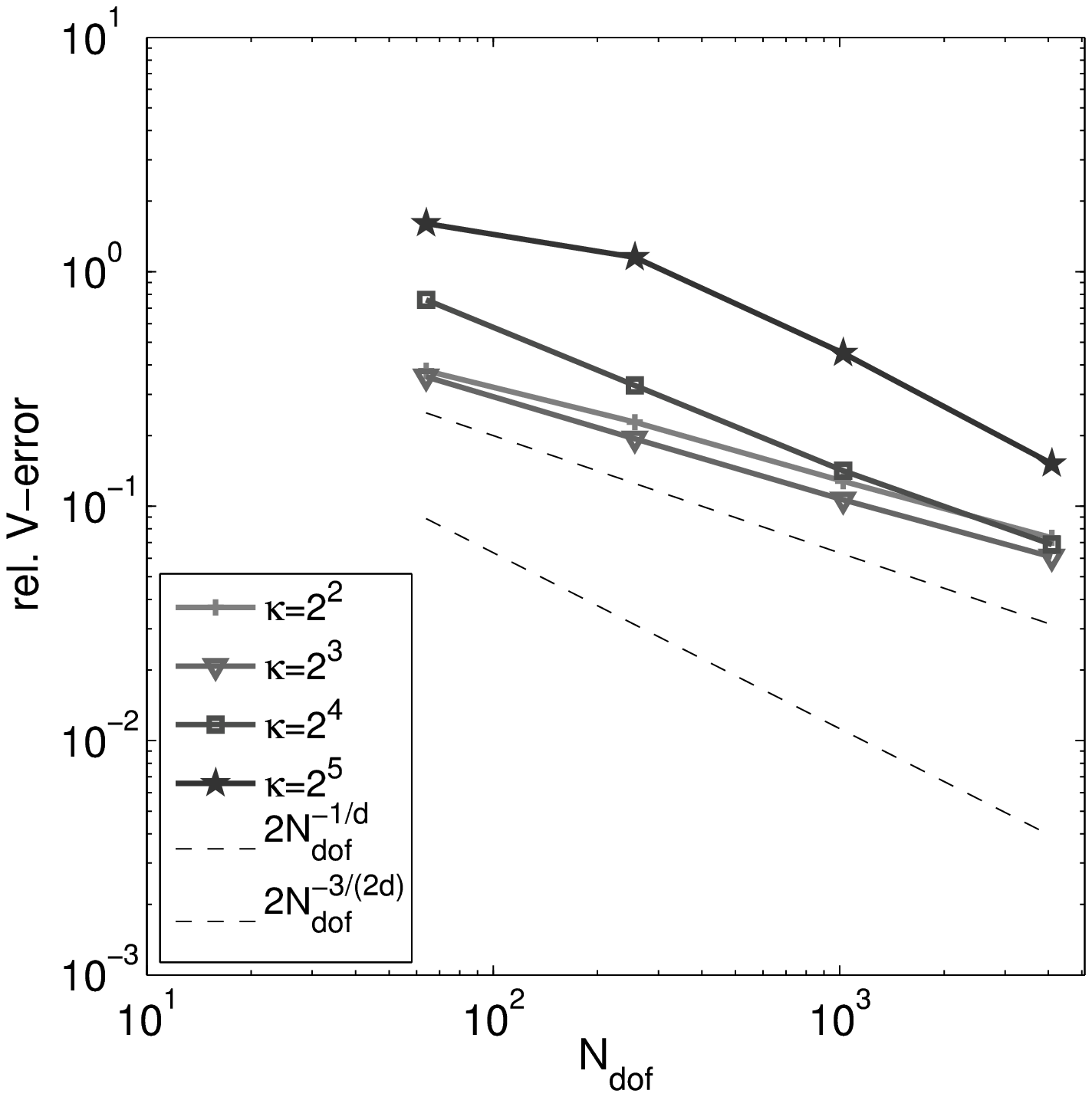}}
\end{center}
\caption{Numerical experiment of Section~\ref{ss:numexp2}: Results for the multiscale method \eqref{e:VclhGalerkin}, a modification based on the trial-test-pairing $(V_H,\Vdlh)$ and standard $P_1$-FEM with several choices of the wave number $\k$ depending on the uniform coarse mesh size $H=N_{\operatorname{dof}}^{-2}$. The reference mesh size $h=2^{-9}$ remains fixed. The oversampling parameter is tied to the coarse mesh size via the relation $\ell=|\log_2 H|$ in (a)-(c). \label{fig:numexp2H}}
\end{figure}

\begin{figure}[tb]
\begin{center}
\subfigure[\label{fig:numexp2k_msfem}Results for multiscale method \eqref{e:VclhGalerkin}.
]{\includegraphics[width=0.49\textwidth]{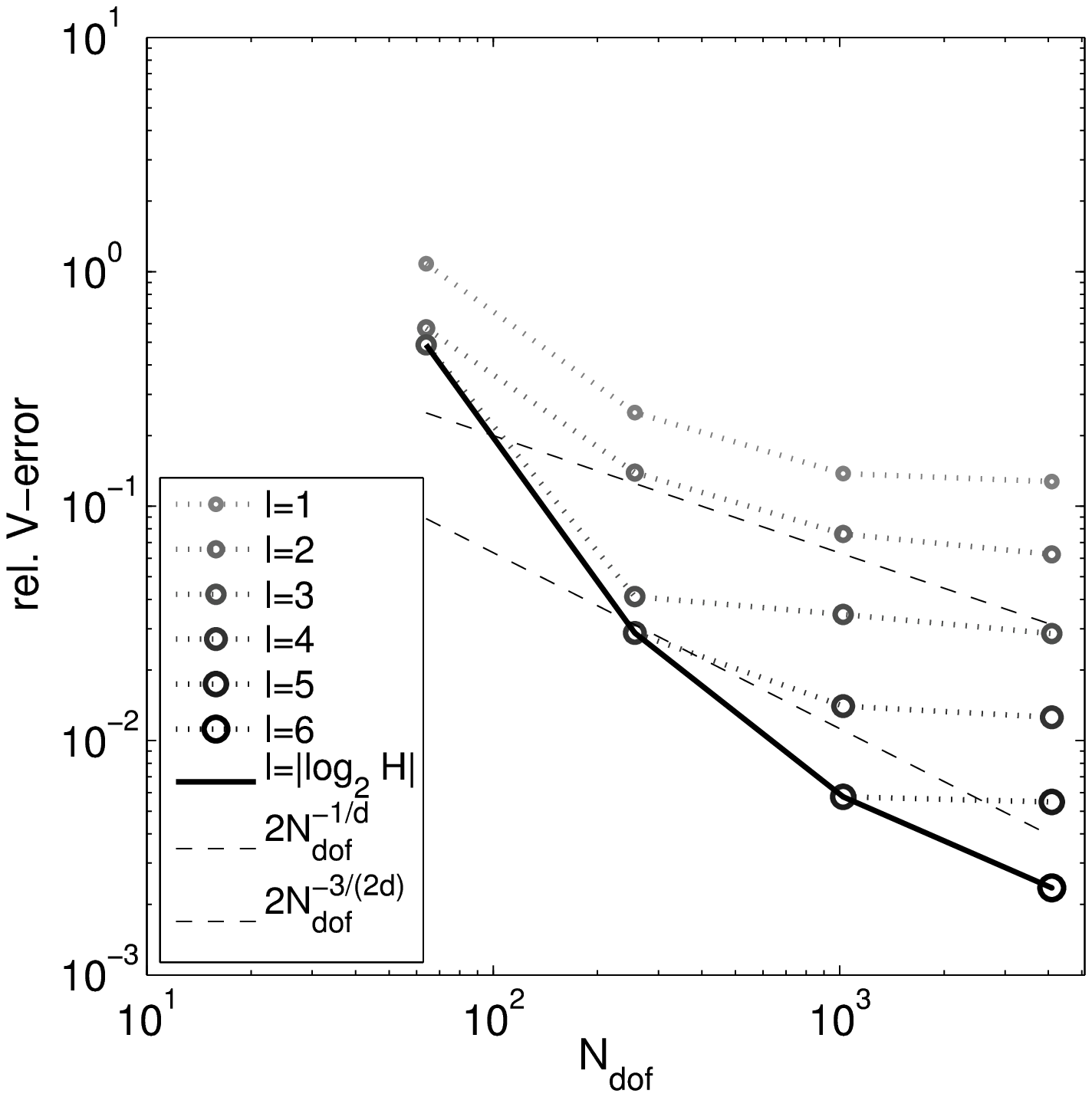}}
\subfigure[\label{fig:numexp2k_msfema}Results for $V$-best-approximation in $\Vplh$.
]{\includegraphics[width=0.49\textwidth]{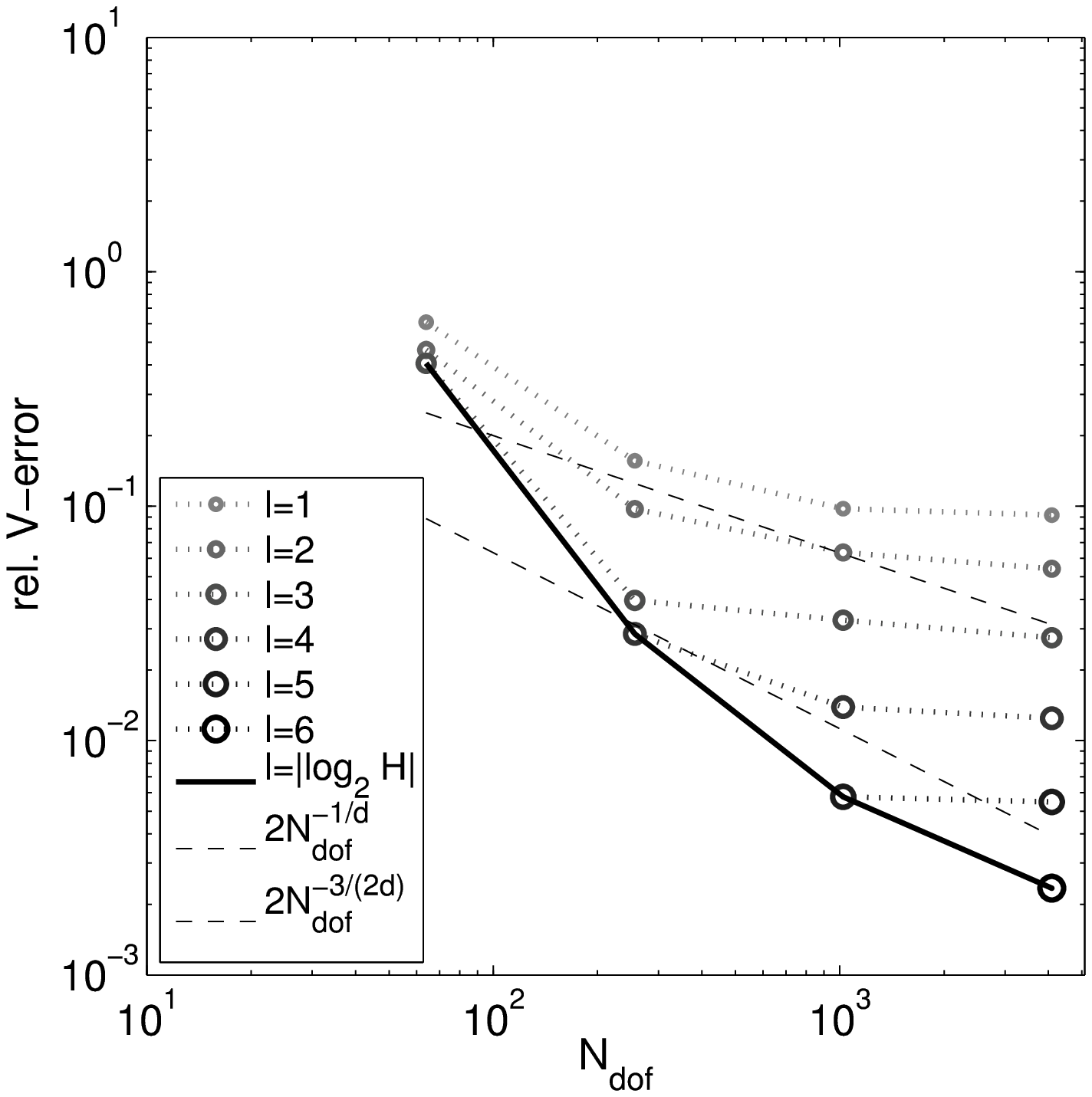}}
\subfigure[\label{fig:numexp2k_msfemb}Results for multiscale Petrov-Galerkin method with trial space $V_H$ and test space $\Vdlh$.
]{\includegraphics[width=0.49\textwidth]{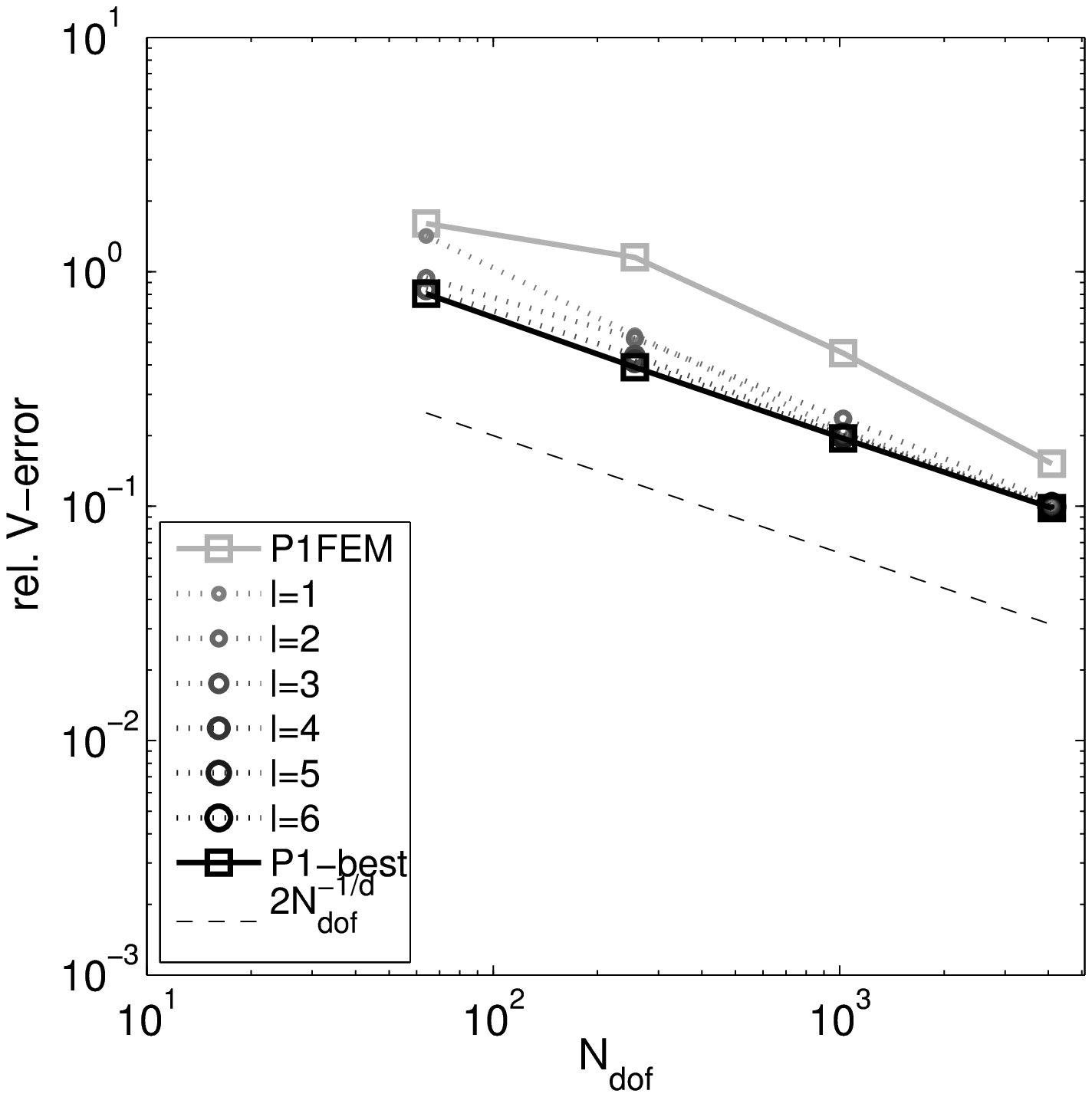}}
\end{center}
\caption{Numerical experiment of Section~\ref{ss:numexp2}: Results for multiscale method \eqref{e:VclhGalerkin} with wave number $\k=2^5$ depending on the uniform coarse mesh size $H=N_{\operatorname{dof}}^{-2}$. The reference mesh size $h=2^{-9}$ remains fixed. The oversampling parameter $\ell$ varies between $1$ and $10$. \label{fig:numexp2k}}
\end{figure}

\newcommand{\etalchar}[1]{$^{#1}$}
\def\cftil#1{\ifmmode\setbox7\hbox{$\accent"5E#1$}\else
  \setbox7\hbox{\accent"5E#1}\penalty 10000\relax\fi\raise 1\ht7
  \hbox{\lower1.15ex\hbox to 1\wd7{\hss\accent"7E\hss}}\penalty 10000
  \hskip-1\wd7\penalty 10000\box7}


\begin{thebibliography}{BCWG{\etalchar{+}}11}

\bibitem[AB14]{MR3247814}
A.~Abdulle and Y.~Bai.
\newblock Reduced-order modelling numerical homogenization.
\newblock {\em Philos. Trans. R. Soc. Lond. Ser. A Math. Phys. Eng. Sci.},
  372(2021):20130388, 23, 2014.

\bibitem[BCWG{\etalchar{+}}11]{betcke}
T.~Betcke, S.~N. Chandler-Wilde, I.~G. Graham, S.~Langdon, and M.~Lindner.
\newblock Condition number estimates for combined potential integral operators
  in acoustics and their boundary element discretisation.
\newblock {\em Numer. Methods Partial Differential Equations}, 27(1):31--69,
  2011.

\bibitem[BP14]{Brown.Peterseim:2014}
D.~{Brown} and D.~{Peterseim}.
\newblock A multiscale method for porous microstructures.
\newblock {\em ArXiv e-prints}, November 2014.

\bibitem[BS00]{BabSau}
I.~M. Babu{\v{s}}ka and S.~A. Sauter.
\newblock Is the pollution effect of the {FEM} avoidable for the {H}elmholtz
  equation considering high wave numbers?
\newblock {\em SIAM Rev.}, 42(3):451--484 (electronic), 2000.
\newblock Reprint of SIAM J. Numer. Anal. {{\bf{3}}4} (1997), no. 6, 2392--2423
  [ MR1480387 (99b:65135)].

\bibitem[BS07]{banjaisauterref}
L.~Banjai and S.~Sauter.
\newblock A refined {G}alerkin error and stability analysis for highly
  indefinite variational problems.
\newblock {\em SIAM J. Numer. Anal.}, 45(1):37--53 (electronic), 2007.

\bibitem[BS08]{MR2373954}
S.~C. Brenner and L.~R. Scott.
\newblock {\em The mathematical theory of finite element methods}, volume~15 of
  {\em Texts in Applied Mathematics}.
\newblock Springer, New York, third edition, 2008.

\bibitem[BY14]{yserentant}
R.~E. Bank and H. Yserentant.
\newblock On the {$H^1$}-stability of the {$L_2$}-projection onto finite
  element spaces.
\newblock {\em Numer. Math.}, 126(2):361--381, 2014.

\bibitem[Car99]{MR1736895}
C.~Carstensen.
\newblock Quasi-interpolation and a posteriori error analysis in finite element
  methods.
\newblock {\em M2AN Math. Model. Numer. Anal.}, 33(6):1187--1202, 1999.

\bibitem[CF00]{MR1807259}
C.~Carstensen and S.~A. Funken.
\newblock Constants in {C}l\'ement-interpolation error and residual based a
  posteriori error estimates in finite element methods.
\newblock {\em East-West J. Numer. Math.}, 8(3):153--175, 2000.

\bibitem[CF06]{feng}
P. Cummings and X. Feng.
\newblock Sharp regularity coefficient estimates for complex-valued acoustic
  and elastic {H}elmholtz equations.
\newblock {\em Math. Models Methods Appl. Sci.}, 16(1):139--160, 2006.

\bibitem[Cia78]{CiarletPb}
P.~G. Ciarlet.
\newblock {\em The finite element method for elliptic problems}.
\newblock North-Holland Publishing Co., Amsterdam-New York-Oxford, 1978.
\newblock Studies in Mathematics and its Applications, Vol. 4.

\bibitem[CV99]{MR1706735}
C.~Carstensen and R.~Verf{\"u}rth.
\newblock Edge residuals dominate a posteriori error estimates for low order
  finite element methods.
\newblock {\em SIAM J. Numer. Anal.}, 36(5):1571--1587 (electronic), 1999.

\bibitem[DGMZ12]{dpg}
L.~Demkowicz, J.~Gopalakrishnan, I.~Muga, and J.~Zitelli.
\newblock Wavenumber explicit analysis of a {DPG} method for the
  multidimensional {H}elmholtz equation.
\newblock {\em Comput. Methods Appl. Mech. Engrg.}, 213/216:126--138, 2012.

\bibitem[DPE12]{ern}
D.~A. Di~Pietro and A. Ern.
\newblock {\em Mathematical aspects of discontinuous {G}alerkin methods},
  volume~69 of {\em Math\'ematiques \& Applications (Berlin) [Mathematics \&
  Applications]}.
\newblock Springer, Heidelberg, 2012.

\bibitem[EG12]{EG12}
O.~G.~Ernst and M.~J.~Gander.
\newblock Why it is difficult to solve {H}elmholtz problems with
              classical iterative methods.
\newblock In {\em Numerical analysis of multiscale problems}, volume~83 of {\em
  Lect. Notes Comput. Sci. Eng.}, pages 325--362. Springer, Heidelberg, 2012.
		
\bibitem[EGMP13]{EGMP13}
D.~{Elfverson}, E.~H. {Georgoulis}, A.~{M{\aa}lqvist}, and D.~{Peterseim}.
\newblock Convergence of a discontinuous galerkin multiscale method.
\newblock {\em SIAM Journal on Numerical Analysis}, 51(6):3351--3372, 2013.

\bibitem[EM12]{MelenkEsterhazy}
S.~Esterhazy and J.~M. Melenk.
\newblock On stability of discretizations of the {H}elmholtz equation.
\newblock In {\em Numerical analysis of multiscale problems}, volume~83 of {\em
  Lect. Notes Comput. Sci. Eng.}, pages 285--324. Springer, Heidelberg, 2012.

\bibitem[FW09]{MR2551150}
X. Feng and H. Wu.
\newblock Discontinuous {G}alerkin methods for the {H}elmholtz equation with
  large wave number.
\newblock {\em SIAM J. Numer. Anal.}, 47(4):2872--2896, 2009.

\bibitem[FW11]{MR2813347}
X. Feng and H. Wu.
\newblock {$hp$}-discontinuous {G}alerkin methods for the {H}elmholtz equation
  with large wave number.
\newblock {\em Math. Comp.}, 80(276):1997--2024, 2011.

\bibitem[For77]{Fortin}
M. Fortin.
\newblock An analysis of the convergence of mixed finite element methods.
\newblock {\em RAIRO Anal. Num\'er.}, 11(4):341--354, iii, 1977.

\bibitem[GP15]{Gallistl.Peterseim:2015}
D.~Gallistl and D.~Peterseim.
\newblock Stable multiscale {P}etrov-{G}alerkin finite element method for high
  frequency acoustic scattering.
\newblock {\em Computer Methods in Applied Mechanics and Engineering}, 295:1--17, 2015.

\bibitem[GGS15]{GGS15}
M.~J.~Gander, I.~G.~Graham and E.~A.~Spence.
\newblock Applying GMRES to the Helmholtz equation with shifted Laplacian preconditioning: what is the largest shift for which wavenumber-independent convergence is guaranteed?
\newblock{\em Numer. Math.} 31(3):567--614, 2015.

\bibitem[HMP14a]{HMP14}
P.~{Henning}, P.~{Morgenstern}, and D.~{Peterseim}.
\newblock {Multiscale Partition of Unity}.
\newblock In M.~Griebel and M.~A. Schweitzer, editors, {\em Meshfree Methods
  for Partial Differential Equations {VII}}, volume 100 of {\em Lecture Notes
  in Computational Science and Engineering}. Springer, 2014.

\bibitem[HMP14b]{HMP13}
P. Henning, A. M{\aa}lqvist, and D. Peterseim.
\newblock A localized orthogonal decomposition method for semi-linear elliptic
  problems.
\newblock {\em ESAIM: Mathematical Modelling and Numerical Analysis},
  48:1331--1349, 9 2014.

\bibitem[HP13]{HP12}
P.~Henning and D.~Peterseim.
\newblock Oversampling for the {M}ultiscale {F}inite {E}lement {M}ethod.
\newblock {\em Multiscale Model. Simul.}, 11(4):1149--1175, 2013.

\bibitem[Het02]{hetmaniukphd}
U.~Hetmaniuk.
\newblock Fictitious domain decomposition methods for a class of partially axisymmetric problems: application to the scattering of acoustic waves. 
\newblock {\em PhD thesis, University of Colorado}, 2002.

\bibitem[Het07]{hetmaniuk}
U.~Hetmaniuk.
\newblock Stability estimates for a class of {H}elmholtz problems.
\newblock {\em Commun. Math. Sci.}, 5(3):665--678, 2007.

\bibitem[HMP11]{perugia}
R.~Hiptmair, A.~Moiola, and I.~Perugia.
\newblock Plane wave discontinuous {G}alerkin methods for the 2{D} {H}elmholtz
  equation: analysis of the {$p$}-version.
\newblock {\em SIAM J. Numer. Anal.}, 49(1):264--284, 2011.

\bibitem[HMP14c]{hiptmair}
R. Hiptmair, A. Moiola, and I. Perugia.
\newblock Trefftz discontinuous {G}alerkin methods for acoustic scattering on
  locally refined meshes.
\newblock {\em Appl. Numer. Math.}, 79:79--91, 2014.

\bibitem[Hug95]{Hughes:1995}
T.~J.~R. Hughes.
\newblock Multiscale phenomena: {G}reen's functions, the
  {D}irichlet-to-{N}eumann formulation, subgrid scale models, bubbles and the
  origins of stabilized methods.
\newblock {\em Comput. Methods Appl. Mech. Engrg.}, 127(1-4):387--401, 1995.

\bibitem[HFMQ98]{MR1660141}
T.~J.~R. Hughes, G.~R. Feij{\'o}o, L. Mazzei, and J.-B.
  Quincy.
\newblock The variational multiscale method---a paradigm for computational
  mechanics.
\newblock {\em Comput. Methods Appl. Mech. Engrg.}, 166(1-2):3--24, 1998.

\bibitem[HS07]{MR2300286}
T.~J.~R.~Hughes and G.~Sangalli.
\newblock Variational multiscale analysis: the fine-scale {G}reen's function,
  projection, optimization, localization, and stabilized methods.
\newblock {\em SIAM J. Numer. Anal.}, 45(2):539--557, 2007.

\bibitem[M{\aa}l11]{Malqvist:2011}
A. M{\aa}lqvist.
\newblock Multiscale methods for elliptic problems.
\newblock {\em Multiscale Model. Simul.}, 9(3):1064--1086, 2011.

\bibitem[Mel95]{melenk_phd}
J.~M. Melenk.
\newblock {\em On generalized finite-element methods}.
\newblock ProQuest LLC, Ann Arbor, MI, 1995.
\newblock Thesis (Ph.D.)--University of Maryland, College Park.

\bibitem[MIB96]{makridakis}
Ch. Makridakis, F.~Ihlenburg, and I.~Babu\v{s}ka.
\newblock Analysis and finite element methods for a fluid-solid interaction
  problem in one dimension.
\newblock {\em Mathematical Models and Methods in Applied Sciences},
  06(08):1119--1141, 1996.

\bibitem[MP14b]{MP14}
A. M{\aa}lqvist and D. Peterseim.
\newblock Localization of elliptic multiscale problems.
\newblock {\em Math. Comp.}, 83(290):2583--2603, 2014.

\bibitem[MP14a]{MP12}
A. M{\aa}lqvist and D. Peterseim.
\newblock Computation of eigenvalues by numerical upscaling.
\newblock {\em Numer. Math.},  130(2):337--361, 2015.

\bibitem[MPS13]{parsania}
J.~M. Melenk, A.~Parsania, and S.~Sauter.
\newblock General {DG}-methods for highly indefinite {H}elmholtz problems.
\newblock {\em J. Sci. Comput.}, 57(3):536--581, 2013.

\bibitem[MS10]{MS10}
J.~M. Melenk and S.~Sauter.
\newblock Convergence analysis for finite element discretizations of the
  {H}elmholtz equation with {D}irichlet-to-{N}eumann boundary conditions.
\newblock {\em Math. Comp.}, 79(272):1871--1914, 2010.

\bibitem[MS11]{MS11}
J.~M. Melenk and S.~Sauter.
\newblock Wavenumber explicit convergence analysis for {G}alerkin
  discretizations of the {H}elmholtz equation.
\newblock {\em SIAM J. Numer. Anal.}, 49(3):1210--1243, 2011.

\bibitem[{Pet}15]{Peterseim2015}
D.~{Peterseim}.
\newblock  Variational multiscale stabilization and the exponential decay of fine-scale correctors. 
\newblock {\em ArXiv e-prints}, 1505.07611, 2015.

\bibitem[PS14]{Peterseim.Scheichl:2014}
D.~{Peterseim} and R.~{Scheichl}.
\newblock Rigorous numerical upscaling at high contrast.
\newblock {\em in preparation}, 2015+.

\bibitem[RHP08]{MR2430350}
G.~Rozza, D.~B.~P. Huynh, and A.~T. Patera.
\newblock Reduced basis approximation and a posteriori error estimation for
  affinely parametrized elliptic coercive partial differential equations:
  application to transport and continuum mechanics.
\newblock {\em Arch. Comput. Methods Eng.}, 15(3):229--275, 2008.

\bibitem[Sau06]{sauterref}
S.~A. Sauter.
\newblock A refined finite element convergence theory for highly indefinite
  {H}elmholtz problems.
\newblock {\em Computing}, 78(2):101--115, 2006.

\bibitem[Sch74]{schatz}
A.~H. Schatz.
\newblock An observation concerning {R}itz-{G}alerkin methods with indefinite
  bilinear forms.
\newblock {\em Math. Comp.}, 28:959--962, 1974.

\bibitem[Szy06]{MR2279449}
D.~B. Szyld.
\newblock The many proofs of an identity on the norm of oblique projections.
\newblock {\em Numer. Algorithms}, 42(3-4):309--323, 2006.

\bibitem[TF06]{MR2219901}
R. Tezaur and C. Farhat.
\newblock Three-dimensional discontinuous {G}alerkin elements with plane waves
  and {L}agrange multipliers for the solution of mid-frequency {H}elmholtz
  problems.
\newblock {\em Internat. J. Numer. Methods Engrg.}, 66(5):796--815, 2006.

\bibitem[Wu14]{Wu01072014}
H. Wu.
\newblock Pre-asymptotic error analysis of cip-fem and fem for the helmholtz
  equation with high wave number. part i: linear version.
\newblock {\em IMA Journal of Numerical Analysis}, 34(3):1266--1288, 2014.

\bibitem[ZMD{\etalchar{+}}11]{Zitelli20112406}
J.~Zitelli, I.~Muga, L.~Demkowicz, J.~Gopalakrishnan, D.~Pardo, and V.M. Calo.
\newblock A class of discontinuous petrov–galerkin methods. part iv: The
  optimal test norm and time-harmonic wave propagation in 1d.
\newblock {\em Journal of Computational Physics}, 230(7):2406 -- 2432, 2011.

\end{thebibliography}
\end{document}